\documentclass[12p]{article}
\usepackage{amsfonts}
\usepackage{a4wide}
\usepackage{amsmath,amscd,amsthm,a4,amssymb}
\usepackage{sectsty}
\usepackage{graphicx}
\graphicspath{{./images/}}
\usepackage{lipsum}
\usepackage[verbose=true,letterpaper]{geometry}
\usepackage{mathtools}
\numberwithin{equation}{section}
\usepackage{tikz}
\usetikzlibrary{3d, calc}

\usepackage{multicol}

\title{\textbf{Best approximation of a three-variable function by sum of one-variable coordinate functions}}
\author{Rashid A.Aliev$^{1}$, Vugar A.Guliyev$^{2}$, Amil F.Jabiyev$^{3}$}

\newtheorem{theorem}{Theorem}[section]
\newtheorem{definition}{Definition}[section]
\newtheorem{lemma}{Lemma}[section]

\newtheorem{remark}{Remark}[section]

\begin{document}
\maketitle
\begin{center}

$^{1,2,3}$ {Baku State University, Baku, Azerbaijan}\\
\vspace{1mm}

$^{1}${Institute of Mathematics and Mechanics, Baku, Azerbaijan}\\
\vspace{1mm}

e-mail: $^{1}$aliyevrashid@mail.ru, $^{2}$ quliyevvuqar.213@gmail.com, $^{3}$ amilcbiyev23@gmail.com
\end{center}
\begin{abstract}
    Consider the following approximation problem of a continuous function of three variables by the sum of three continuous functions of one variable:
    \[ E(f,\Omega)=\inf||f(x,y,z)-\phi(x)-\psi(y)-\omega(z)||_{\infty}=? \]
    where $f(x,y,z)$ is a given continuous function defined on $\Omega=[0,1]^3$ and the infimum runs over all triplets of continuous functions $\phi(x),\psi(y),\omega(z)$ defined on the unit interval $[0,1]$. In this paper, we will prove a formula to calculate the error $E(f,\Omega)$ under certain conditions.
\end{abstract}

\begin{quote}\small
\end{quote}

\begin{quote}\small
\end{quote}

\smallskip

\section{Introduction}

\subsection{Problem statement in a general form and previous results}

Let \(f:X \to \mathbb{R}\) be a continuous function of \(n\) variables, where \(X=X_1\times X_2 \times\dots X_n\) is the Cartesian product of \(n\) compact Hausdorff spaces. Consider the approximation of \(f\) by sums of \(g_i, i=1,\dots,n\), where \(g_i\) is a continuous function on \(X_i\) for \(i=1,\dots,n\). In other words, the problem asks for 
\begin{equation} \label{e: main-problem}
    E(f,X)=\inf_{(g_1,g_2,...,g_n)}||f(x_1,x_2,\dots,x_n)-\sum_{i=1}^ng_i(x_i)||_{C(X)}=?
\end{equation}
where the infimum runs over all \(n\)-tuples of functions \((g_1,g_2,...,g_n)\) satisfying \(g_i \in C(X_i),i=1,\dots,n\).

In general, this problem is related to Hilbert's 13th problem. Using A.S.Kronrod's \cite{Kronrod} results, A.N. Kolmogorov \cite{Kolmogorov_1956} showed that any function $f \in C([0,1]^{n})$ is representable as a superposition of continuous functions of three variables. Developing this method, V.I.Arnold \cite{Arnold_1957}, \cite{Arnold_1959} proved that each continuous function of three variables can be represented by a superposition of continuous functions of two variables. This is a negative answer to Hilbert's conjecture. Further, A.N.Kolmogorov \cite{Kolmogorov_1957} proved the following remarkable, stronger theorem in a simple, elementary and elegant way:

	\textbf{Kolmogorov’s Theorem.} There exist increasing functions $\phi_{pq} \in C([0,1])$, $p=1,...,n$, $q=1,...,2n+1$, such that an arbitrary given function
\[f(x_1,...,x_n) \in C([0,1]^n) \]
can be represented in the form
\[f(x_1,...,x_n)=\sum_{q=1}^{2n+1}g_q \left(\sum_{p=1}^n \phi_{pq}(x_p)  \right),  \]
where $g_q \in C([0,1])$ depend on $f$.

Thus, every continuous function on the $n$ variables can be represented by a superposition of continuous functions of one variable and the simplest function of two variables, namely, the sum of variables.

In addition to the problems of finding a precise expression for a given function in terms of a combination of functions of fewer variables, it is natural to consider the best approximation of functions of several variables by combinations of functions of fewer variables. These problems arise in various areas of mathematics, for example, in solving integral equations (see \cite{Collatz}), functional equations (see \cite{Buck}), in saving information (see \cite{Diliberto_Straus}), in scaling matrices (see \cite{Golitschek_1980}, \cite{Golitschek_1982}), in isogeometric analysis (see \cite{Bosy}), etc. 

In the space $C([0,1]^n)$, this problem was first studied by S. P. Diliberto and E. G. Straus \cite{Diliberto_Straus} (the problem was proposed by RAND Corporation for the most efficient distribution of data in the small memory of the computers of the day). The same problem was investigated by Yu. P. Ofman \cite{Ofman} for bounded bivariant functions, and for more general sets by D. E. Marshall and A. G. O'Farrell in \cite{Marshall_Farrel}. In \cite{Diliberto_Straus}, Diliberto and Straus gave a “leveling algorithm" to find the approximation error (1.1) for \(n=2\). 
The addressed problem, together with the methods developed from its analysis, finds important applications in approximation by subalgebras and, in particular, in RBF approximation, ridge function approximation, and neural network models. Recent papers (see \cite{Akhmedov, Asgarova_2025, Asgarova_Huseynli, Asgarova_Maharov, Aysu}) improve the two-dimensional theory in connection with these directions. Moreover, based on the “levelling" algorithm, the ADI (alternating-direction-implicit) iteration method for numerical solution of differential equations is developed (see \cite{Tang, Wachspress_1994,Wachspress_2013}).

In the same paper, Diliberto and Straus stated that a similar algorithm will work for \(n \ge 3\). However, G. Aumann in his work \cite{Aumann} has shown that while the proof in \cite{Diliberto_Straus} was complete for case \(n=2\), the generalization of the algorithm suggested for \(n \ge 3\) was incorrect and proved a necessary and sufficient condition on the class of approximating functions that the “leveling algorithm" works. A direct counterexample refuting the “leveling algorithm" in \(n \ge 3\) was given by V.A. Medvedev \cite{Medvedev}. This example shows that the levelling algorithm, nor its small variations, can converge to a best approximation from the sum of three subspaces (see \cite{Pinkus}).

Another solution approach for \(n \ge 3\) to the problem starts with the paper \cite{Golomb} by M. Golomb, where an important duality formula involving functionals defined on projection-cycles for the problem \eqref{e: main-problem} was stated. However, later in \cite{Marshall_Farrel} by D.E. Marshall and A.G. O'Farrell, it was shown that the proof of the formula given in \cite{Golomb} also has a gap. In \cite{Marshall_Farrel}, using the ergodic method, the lightning-bolt principle is also proven, which corresponds to the two-dimensional version of Golomb's formula. It is worth remarking that the results obtained in the same paper in the two-dimensional case do not extend to sums such as $f (x)+g(y)+h(z)$ on subsets of \(\mathbb{R}^3\). There is no way to generate the annihilators of a sum of three algebras using any analogue of the lightning bolts. This is related to the fact that ergodic theory is essentially limited to actions of groups that have subexponential growth.  This shows that there are essentially new phenomena that arise with the sum of three subalgebras of \(C(X)\), as opposed to two (see \cite{O'Farrel}).

Future, with the help of the results of S.Ya. Khavinson \cite{Khavinson} and K.G. Navada \cite{Navada}, V. Ismailov \cite{Ismailov_2009} proved that the formula given by M. Golomb is valid in a stronger form. Although this formula itself involves finding a supremum over an infinite number of functionals and there is no closed formula to find the error explicitly, even with no additional conditions.

The first finite formula for the two-dimensional case is established by T.J.Rivlin and R.J.Sibner in the paper \cite{Rivlin_Sibner}, under the condition \(f_{xy} \ge 0\), 
\[  E(f,[0,1]^2)=\frac{1}{4}[f(1,1)-f(1,0)-f(0,1)+f(0,0)]. \]
L. Flatto gave a different proof of the last formula in his work \cite{Flatto}, where a general result for the approximation of an \(n\) variable function by \(n-1\) variable functions was suggested under similar conditions.
Some other notable improvements in this direction have been made in \cite{Muhammad} \cite{Ismailov_2021}, where generalisations of the above formula in stair-like polygons have been proven. However, no finite formula has been proven for other domains, even for the simplest ones as the circle and triangle.

In case $n \ge 3$ (even with additional conditions), there is no finite formula and converging algorithm to calculate the error $E(f, X)$.
In the present paper, under conditions $\Delta_x\Delta_y f\ge 0$, $\Delta_x\Delta_z f\ge 0$, $\Delta_y\Delta_z f\ge 0$, we will prove a finite formula to calculate error (1.1) in  $[0,1]^3$ (See Theorem \ref{t:Theorem 1}). 

The obtained results and methods proposed in this paper can be used in or extended to approximation by subalgebras and, in particular, in RBF approximation, ridge function approximation, and neural network models in the three-dimensional case. Moreover, the same ideas could lead to the construction of a three-dimensional levelling-type algorithm to find the best approximation functions, too. Such an algorithm would allow the construction of more efficient versions of existing numerical solution methods. As an example, the three-dimensional ADI iteration method itself proposed in \cite{Wachspress_1994} uses the two-dimensional levelling algorithm due to the lack of a converging three-dimensional alternative.

\subsection{Basic definitions and formulas}
\bigskip
We start with some definitions that will be used throughout this paper. Most of the definitions in this subsection are taken from the paper \cite{Ismailov_2009} (also can be found in the monograph \cite{Ismailov_2021}).
\begin{definition}
    A set of points \(p=(x_1,x_2,...,x_m), \ x_i \in X\) where the set \(X\) is the Cartesian product \(X=X_1\times X_2 \times\dots X_n\), is called a \textbf{projection cycle} if there exist some nonzero real numbers \(\lambda_1,\lambda_2,..\lambda_m\) such that
    \begin{equation}\label{e:projection}
        \sum_{\pi_j(x_i)=\xi}\lambda_i=0, \ \forall  \xi \in X_j \ \text{ and } \ \forall j, \  0\le j\le n,
    \end{equation}
    where \(\pi_j(x)\) is the projection of \(x\) into \(X_j\) and \(\xi\) is an arbitrary element of the set \(X_j\). 
\end{definition}
The definition of a projection cycle goes back to M.Golomb \cite{Golomb}. The above-given definition is an equivalent form of the Golomb's definition taken from \cite{Ismailov_2009}. 
%\begin{definition} \label{d:weak-projection-cycle}
    %A set of points \(p=(x_1,x_2,...,x_m), \ x_i \in X\) where the set \(X\) is the Cartesian product \(X=X_1\times X_2 \times\dots X_n\), is called a \textbf{weak projection cycle} if there exist some real numbers \(\lambda_1,\lambda_2,..\lambda_m\) such that, at least one \(\lambda_i\) is nonzero and,
    %\begin{equation}\label{e: weak projection}
        %\sum_{\pi_j(x_i)=\xi}\lambda_i=0, \forall j,  0\le j\le n \text{ and } \forall \xi \in X_j,
    %\end{equation}
    %where \(\pi_j(x)\) is the projection of \(x\) into \(X_j\) and \(\xi\) is an arbitrary element of the set \(X_j\). 
%\end{definition}
%The difference between a projection cycle and a weak projection cycle is that the second may contain zero weighted points. Thus, any projection-cycle is a weak projection cycle, but the converse is not true in general. For example, the set of points \(\{(0,0),(0,1),(1,0),(1,1),(2,2)\}\) in \(\mathbb{R}^2\) is a weak projection cycle but not a projection cycle.
\begin{definition}
    A projection cycle \(p=(x_1,x_2,...,x_m)\) together with a vector \(\lambda=(\lambda_1,\lambda_2,..\lambda_m)\) satisfying the equation \eqref{e:projection} is called a \textbf{projection cycle-vector} and denoted by \((p,\lambda)\). The number of points of \(p\) is called the length of the projection-cycle vector \((p,\lambda)\). And \(\lambda_i=\lambda(x_i)\) is called the weight at point \(x_i\).
\end{definition}
%When we are given a weak projection cycle \(p\) and a non-zero vector \(\lambda\) satisfying the equation \eqref{e: weak projection}, by \((p,\lambda)\) we mean the projection cycle-vector based on the points of \(p\) corresponding to nonzero components of \(\lambda\).
\begin{remark}
    For \(n=3\), the left-hand side of equation \eqref{e:projection} becomes the sum over planes parallel to one of the coordinate planes. Thus, if a projection-cycle has a point on a plane parallel to one of the coordinate planes, then it has at least one other point with an opposite signed weight on the same plane.
\end{remark}
\begin{definition}
    A projection cycle is called a \textbf{minimal projection cycle} if it has no proper subset that is a projection cycle. A projection cycle-vector \((p,\lambda)\) is called a \textbf{minimal projection cycle-vector} if the set of points of \(p\) corresponding to nonzero components of \(\lambda\) is a minimal projection cycle.
\end{definition}
\begin{definition}
    A projection cycle-vector \((p,\lambda)\) is called orthonormal if
\[     \sum|\lambda_i|=1.    \]
\end{definition}

The concept of the minimal projection cycle was defined by name \textit{loop} in \cite{Navada} by K.G. Navada.

The following theorem expresses Golomb's formula (proposed in \cite{Golomb}), in a revisited and improved form, from \cite{Ismailov_2009} (See: Theorem 2.8 in \cite{Ismailov_2009} or Theorem 3.37 in \cite{Ismailov_2021}).
\begin{theorem}[Golomb's Formula] The error of approximation in \eqref{e: main-problem} can be given as 
    \begin{equation} \label{e:Golomb}
   E(f,X)=\sup_{(p,\lambda)} \frac{\sum_{}\lambda_if(x_i)}{\sum|\lambda_i|}. 
\end{equation}
where the supremum runs over all minimal projection cycles-vectors on the set \(X\).
\end{theorem}
\begin{remark}
    For simplicity, we use the term ‘‘projection cycle-vector’’ instead of ‘‘projection cycle, vector pair’’ used in \cite{Ismailov_2009} and \cite{Ismailov_2021}.
\end{remark}
In what follows, the following notations will be used.
As usual \(\Delta_{(x,y,z)} f\) is defined as
\[
\Delta_{(x,y,z)} f(u,v,w)= f(u+x,v+y,w+z)-f(u,v,w).
\]
For simplicity, we denote
\[
\Delta_x=\Delta_{(x,0,0)}, \Delta_y=\Delta_{(0,y,0)},\Delta_z=\Delta_{(0,0,z)}.
\]
For any finite sets \(S\), the set \([S]\) denotes the set of all projection cycle-vectors, which, each point belongs to the set \(S\) and \([S]_{min}\) denote the set of all minimal projection cycle-vectors \((p,\lambda) \in [S]\).

\section{Three variable case}
\bigskip

Let \(\Omega=[0,1]^3\) denote the unit cube, and let \(f:\Omega\rightarrow \mathbb{R}\) be a given continuous function of three variables. Consider the problem:
\[
    E(f,\Omega)=\inf||f(x,y,z)-\phi(x)-\psi(y)-\omega(z)||_{C(X)}=?
\]
where the infimum runs over all triplets of continuous functions \(\phi(x),\psi(y),\omega(z)\) defined on the unit interval \([0,1]\).
\\
Assume \(f\) satisfies the following conditions at any point of \(\Omega\)
\begin{equation} \label{e: delta class definition}
    \begin{cases}
        \Delta_x\Delta_y f\ge 0\\
        \Delta_x\Delta_z f\ge 0\\
        \Delta_y\Delta_z f\ge 0.
    \end{cases}
\end{equation}
\\
If \(f\) has mixed derivatives with respect to any pair of variables at any point of \(\Omega\), then the inequalities \eqref{e: delta class definition} are equivalent to
\begin{equation} \label{e:definition of the class}
    \begin{cases}
        f_{xy}\ge 0\\
        f_{yz}\ge 0\\
        f_{xz}\ge 0.
    \end{cases}
\end{equation}
\\
We start by applying Golomb's Formula \eqref{e:Golomb} and get:
\begin{equation} \label{e:Golomb-applied}
   E(f,\Omega)=\sup_{(p,\lambda)} \frac{\sum_{}\lambda_if(x_i)}{\sum|\lambda_i|}. 
\end{equation}
where the supremum runs over all minimal projection cycle-vectors \((p,\lambda)\) on the set \(\Omega\). To proceed further, we will search the supremum \eqref{e:Golomb-applied} as
\begin{equation} \label{e: supremum-maximum}
   E(f,\Omega)=\sup_{(p,\lambda)} \frac{\sum_{}\lambda_if(x_i)}{\sum|\lambda_i|}=\sup_{A,B,C}  \left({\max_{(p,\lambda)\in [A\times B\times C]} \frac{\sum_{}\lambda_if(x_i)}{\sum|\lambda_i|}}\right),
\end{equation}
where the last supremum runs over all triples of finite sets \(A,B,C \subset [0,1]\).
Obviously, for any  projection cycle-vector \((p,\lambda)\), there exists at least one triple of sets \(A,B,C \subset [0,1]\) such that \((p,\lambda) \in [A\times B\times C]\). On the other hand, choosing the sets \(A,B,C\) finite ensures that the above maximum over \([A\times B\times C]\) is attained. Hence, the relation \eqref{e: supremum-maximum} is well-defined. 

Before proving our main results, we have to prove some facts. From now on will proceed with the following stricter version of the inequalities \eqref{e: delta class definition}
\begin{equation} \label{e:strict-class-definition}
    \begin{cases}
        \Delta_x\Delta_y f> 0\\
        \Delta_x\Delta_z f> 0\\
        \Delta_y\Delta_z f> 0.
    \end{cases}
\end{equation} 
Later, we will see that we can pass back to the original condition \eqref{e: delta class definition}.

\subsection{Optimal Projection Cycle-Vectors  on a Finite Set} 

We start by investigating properties of projection cycle-vectors that give the following maximum
\begin{equation}\label{e: maximum-over-a-finite-set}
    \max_{(p,\lambda)\in [A\times B\times C]} \frac{\sum_{}\lambda_if(x_i)}{\sum|\lambda_i|},
\end{equation}
where \( A,B,C \subset [0,1]\) are the sets chosen arbitrary with \(\{0,1\}\subset A \cap B \cap C\) and \(f \in C(\Omega)\) is a function satisfying the conditions \eqref{e: delta class definition} or \eqref{e:strict-class-definition}.
\\
Let us give some further definitions and auxiliary facts.
\begin{definition}
    Let's call two points \((\xi_1,\gamma_1,\zeta_1)\) and \((\xi_2,\gamma_2,\zeta_2)\) of the space \(\mathbb{R}^3\) \textbf{well-ordered} if either \[\xi_1\ge \xi_2,\gamma_1 \ge \gamma_2 ,\zeta_1 \ge \zeta_2  \
    \text{ or } \
    \xi_1\le \xi_2,\gamma_1 \le \gamma_2 ,\zeta_1 \le \zeta_2.\]
    In other words, two points are called well-ordered if all components of one of the points are not smaller than the corresponding coordinates of the other point.
\end{definition}
If two points\((\xi_1,\gamma_1,\zeta_1)\) and \((\xi_2,\gamma_2,\zeta_2)\) are not well-ordered, then, without loss of generality, we can assume that:
    \[\xi_1\ge \xi_2, \ \gamma_1 > \gamma_2, \ \zeta_1 <\zeta_2.\]
All other versions can be transferred into this, either by changing the names of variables or the names of points.

Now we present an important lemma, which is a direct consequence of the inequalities \eqref{e:strict-class-definition}.
\begin{lemma} \label{main-func-ineq}
For any function \(f \in C(\Omega)\) satisfying \eqref{e:strict-class-definition} and any two points \(x_1=(\xi_1, \gamma_1, \zeta_1),x_2=(\xi_2, \gamma_2, \zeta_2) \in \Omega\) the inequality
\begin{equation} \label{e:ordering}
    f(\xi_1, \gamma_1, \zeta_1)+f(\xi_2, \gamma_2, \zeta_2) \le f\left((\xi_1,\xi_2),( \gamma_1,\gamma_2)(\zeta_1,\zeta_2)\right) +f\left([\xi_1,\xi_2],[\gamma_1,\gamma_2],[\zeta_1,\zeta_2)]\right)
\end{equation}
holds, where we denote \((a,b)=\min(a,b)\) and \([a,b]=\max(a,b)\). The equality happens iff \(x_1\) and \(x_2\) are well-ordered.
\end{lemma}
\begin{proof} 
    If \(x_1\) and \(x_2\) are well-ordered, then the expressions on the left- and right-hand sides become the same, and the equality holds.\\
    If \(x_1\) and \(x_2\) are not well-ordered, without loss of generality we can assume that \(\xi_1 \ge \xi_2\), \(\gamma_1 > \gamma_2\) and \(\zeta_1 < \zeta_2\).
    Thus, we have
    \[
    (\xi_1,\xi_2)=\xi_2, (\gamma_1,\gamma_2)=\gamma_2,(\zeta_1,\zeta_2)=\zeta_1
    \text{
     and 
    }
    [\xi_1,\xi_2]=\xi_1, [\gamma_1,\gamma_2]=\gamma_1,[\zeta_1,\zeta_2]=\zeta_2.
    \]
Let's name \(x_1'=(\xi_1, \gamma_1, \zeta_2),x_2'=(\xi_2, \gamma_2, \zeta_1)\), \(x_1''=(\xi_1, \gamma_2, \zeta_1) \) and \( x_2''=(\xi_1, \gamma_2, \zeta_2)\).
The inequality is equivalent to
    \[
    f(x_1') + f(x_2') - f(x_1) - f(x_2)
    \]
    \[
    = \left( f(x_2') + f(x_2'') - f(x_1'') - f(x_2)\right) + \left(\left( f(x_1'') + f(x_1') - f(x_2'') - f(x_1)\right)\right) >0,
    \]
which, directly follows from \eqref{e:strict-class-definition}(See: Figure 1).
\end{proof}
\begin{figure}[ht]
\centering
\begin{tikzpicture}[scale=0.3, line join=round, line cap=round]

\draw[->, thick, red] (0,0,13) -- (10,0,13) node[below left] {$x$};
\draw[->, thick, red] (0,0,13) -- (0,10,13) node[left] {$z$};
\draw[->, thick, red] (0,0,13) -- (0,0,-3) node[left] {$y$};

\coordinate (xi') at (2,3,11);
\coordinate (xj') at (7,3,4);
\coordinate (xi) at (2,7,11);
\coordinate (xj) at (7,7,4);

\coordinate (xip) at (7,7,11);
\coordinate (xjp) at (7,3,11);

\fill[black!30, opacity=0.111] (0,0,0)--(9,0,0)--(9,0,13) -- (0,0,13);
\fill[black!30, opacity=0.204] (0,0,0)--(0,8,0)--(0,8,13) -- (0,0,13);
\fill[black!30, opacity=0.3] (0,0,0)--(0,8,0)--(9,8,0)--(9,0,0);
% Draw the two parallelograms (sides of a prism)
\fill[cyan!30, opacity=0.313] (xi) -- (xi') -- (xjp) -- (xip) -- cycle;
\fill[red!30, opacity=0.313] (xj') -- (xj) -- (xip) -- (xjp) -- cycle;

\draw[thick, green!60!black] (xj') -- (xjp);
\draw[thick, green!60!black] (xjp) -- (xip);
\draw[thick, green!60!black] (xjp) -- (xi');

\draw[thick, green!60!black] (xj) -- (xip);
\draw[thick, green!60!black] (xi) -- (xip);

\draw[thick, green!60!black] (xi') -- (xi);
\draw[thick, green!60!black] (xj') -- (xj);

\filldraw[red] (xip) circle (3pt)
node[above left] {$x_2''$};
\filldraw[red] (xjp) circle (3pt)
node[right] {$x_1''$};

\filldraw[black] (xi') circle (3pt) node[left] {$x_2'$};
\filldraw[black] (xi) circle (3pt) node[left] {$x_2$};
\filldraw[black] (xj') circle (3pt) node[right] {$x_1$};
\filldraw[black] (xj) circle (3pt) node[right] {$x_1'$};

\end{tikzpicture}
         \label{f:Lemma 2.1}
         \caption{Geometric interpretation of the proof of Lemma \ref{main-func-ineq}}
\end{figure}

\begin{definition}
    A projection cycle-vector \((p^0,\lambda^0) \in [S]\) is called \textbf{optimal} for the function \(f\) on the finite set \(S\) if 
    %\(p^0\) is a minimal projection cycle and
    \[
    \frac{\sum_{}\lambda^0_if(x^0_i)}{\sum|\lambda^0_i|}=\max_{(p,\lambda)\in [A\times B\times C]} \frac{\sum_{}\lambda_if(x_i)}{\sum|\lambda_i|}.
    \]
    The set of all \textbf{optimal projection cycle-vectors} for \(f\) on \(S\) is denoted by \([S]_f\).
\end{definition}
Obviously, for any \((p,\lambda) \in [S]_f\) we have
\begin{equation} \label{e:orthonormal sum}
    \frac{\sum_{}\lambda_if(x_i)}{\sum|\lambda_i|} \ge 0,
\end{equation}
as otherwise, taking \((p,-\lambda)\) will contradict the optimality of \((p,\lambda)\). And, moreover since \([S]\) is finite, \([S]_f\) is always non-empty.

\begin{definition}
    The sum of two projection cycles-vectors is a projection cycle-vector on the union of the points of the given projection cycle-vectors, with weight at each point equal to the sum of the weights of the given projection cycle-vectors on that point. If we encounter any points with a weight of zero, we would ignore them.
\end{definition}
We can express the above definition by formulas as follows: 
\[(p^1,\lambda^1)+(p^2,\lambda^2)=(p,\lambda),\] where
\(
p=\left( p^1 \cup p^2\right) \backslash p^*
\),  $p^*=\{q:q \in p^1 \cap p^2, \lambda^1(q)+\lambda^2(q)=0 \}$
and
\begin{equation}\label{sum of two projection-cycle-vectors}
\lambda(x)=\begin{cases}
    \lambda^1(x), \hspace{45pt} x \in p^1\backslash p^2 \\
    \lambda^1(x)+ \lambda^2(x), \hspace{8pt} x \in \left(p^1 \cap p^2 \right) \backslash p^*\\
    \lambda^2(x), \hspace{43pt}  x \in p^2\backslash p^1 \\
\end{cases}
\end{equation}
for any two projection cycle-vectors $(p^1,\lambda^1)$ and $(p^2,\lambda^2)$. 

Clearly, if both \((p^1,\lambda^1) \in [S] \) and \((p^2,\lambda^2) \in [S]\) for some finite set \(S\), then \[(p^1,\lambda^1)+(p^2,\lambda^2) \in [S]\] since \(p^1 \subset S\) and \(p^2 \subset S\) is equivalent to \(p^1 \cup p^2 \subset S\).
Moreover for any function \(f\) defined on \([0,1]^3\), we have
\[
\sum_{x_i \in p}\lambda_if(x_i)=\sum_{x_j \in p^1}\lambda^1_jf(x_j)+\sum_{x_k \in p^2}\lambda^2_kf(x_k).
\]
In fact, we will never encounter a situation of adding two projection cycle-vectors of the form \((p, \lambda)\) and \((p, -\lambda)\). However, for the sake of correctness, we can still denote this sum as a \textbf{vanishing} projection cycle-vector. That is a projection cycle, with no points and no weights, and will not be considered an element of \([S]\), for any set \(S\).

The following lemma demonstrates our way of constructing a “more optimal" projection-cycle vector from the given “non-optimal" one.
\begin{lemma} \label{lemma-denominators}
    Let \((p^1,\lambda^1)\) and \((p^2,\lambda_\varepsilon)\) be two projection cycle vectors, where \(\lambda_\varepsilon\) is a vector of length \(m\) such that each component is \(\varepsilon\) or \(-\varepsilon\) for some positive real number \(\varepsilon\). Let \(k\) be the number of common points of \(p^1\) and \(p^2\), such that the corresponding components of the vectors \(\lambda^1\) and \(\lambda_\varepsilon\) have opposite signs. If \(k \ge \frac{m}{2} \), then for
    \[
    0 < \varepsilon \le \min_{\substack{x \in p^1 \cap p^2 \\ \lambda^1(x)\lambda_\varepsilon(x)<0}}|\lambda^1(x)|
    \] we have
    \begin{equation}
        \sum_{x \in p}|\lambda(x)| \le \sum_{x \in p^1}|\lambda^1(x)|,
    \end{equation}
    where \((p,\lambda)=(p^1,\lambda^1)+(p^2,\lambda_\varepsilon)\).
\end{lemma}
\begin{proof}
    By the equation \eqref{sum of two projection-cycle-vectors} we have
    \begin{equation*}
    \sum_{x \in p}|\lambda(x)|-\sum_{x \in p^1}|\lambda^1(x)|=\sum_{x \in p^1 \cap p^2}\left(|\lambda^1(x)+\lambda_\varepsilon(x)|-|\lambda^1(x)|\right)+\sum_{x \in p^2\backslash p^1}\varepsilon.
    \end{equation*}
    The first summand can be written as
    \[
    \sum_{x \in p^1 \cap p^2}\left(|\lambda^1(x)+\lambda_\varepsilon(x)|-|\lambda^1(x)|\right)=
    \]
    \[
    =\sum_{\substack{x \in p^1 \cap p^2 \\ \lambda_1(x)\lambda_\varepsilon(x)<0}}\left(|\lambda^1(x)+\lambda_\varepsilon(x)|-|\lambda^1(x)|\right)+\sum_{\substack{x \in p^1 \cap p^2 \\ \lambda_1(x)\lambda_\varepsilon(x)>0}}\left(|\lambda^1(x)+\lambda_\varepsilon(x)|-|\lambda^1(x)|\right).
    \]
    By the triangle inequality, we have 
    \[
    \sum_{\substack{x \in p^1 \cap p^2 \\ \lambda_1(x)\lambda_\varepsilon(x)>0}}\left(|\lambda^1(x)+\lambda_\varepsilon(x)|-|\lambda^1(x)|\right) \le \sum_{\substack{x \in p^1 \cap p^2 \\ \lambda_1(x)\lambda_\varepsilon(x)>0}}|\lambda_\varepsilon(x)|=\sum_{\substack{x \in p^1 \cap p^2 \\ \lambda_1(x)\lambda_\varepsilon(x)>0}}\varepsilon.
    \]
    On the other hand, for  \[ 0 < \varepsilon \le \min_{\substack{x \in p^1 \cap p^2 \\ \lambda_1(x)\lambda_\varepsilon(x)<0}}|\lambda^1(x)|\] we have
    \[
    \sum_{\substack{x \in p^1 \cap p^2 \\ \lambda_1(x)\lambda_\varepsilon(x)<0}}\left(|\lambda^1(x)+\lambda_\varepsilon(x)|-|\lambda^1(x)|\right)=-\sum_{\substack{x \in p^1 \cap p^2 \\ \lambda_1(x)\lambda_\varepsilon(x)<0}} \varepsilon.
    \]
    Thus, we have
    \[
    \sum_{x \in p}|\lambda(x)|-\sum_{x \in p^1}|\lambda^1(x)| \le -\sum_{\substack{x \in p^1 \cap p^2 \\ \lambda_1(x)\lambda_\varepsilon(x)<0}} \varepsilon + \sum_{\substack{x \in p^1 \cap p^2 \\ \lambda_1(x)\lambda_\varepsilon(x)>0}}\varepsilon + \sum_{x \in p^2\backslash p^1}\varepsilon \le 0.
    \]
    Where the last inequality follows from the fact that the number of summands with a negative sign is \(k\), the number of all summands is \(m\), and \(k \ge \frac{m}{2}\)
\end{proof}

The next theorem deals with the geometric structure of optimal projection cycles. 
\begin{theorem} \label{t: Theorem 2.2}
    Let \((p,\lambda)\) be a projection cycle-vector, that is optimal for a function \(f \in C(\Omega)\) satisfying \eqref{e:strict-class-definition} on the set \(S=S_x\times S_y\times S_z\) for some finite sets \(S_x, S_y, S_z \subset [0,1]\) with \(\{0,1\}\subset S_x \cap S_y \cap S_z\). Then the following propositions hold:
    \begin{enumerate}

        \item[(a)] If \(\lambda_k,\lambda_j\) are both positive for some indices \(k,j\) then, \(x_k\) and \(x_j\) are well-ordered,
        \item[(b)] Let us call a plane on \(\mathbb{R}^3\) an interior plane of the unit cube if it is parallel to one of the faces of the unit cube and passes through an interior point of the cube. If \(p\) contains a well-ordered pair of points  \(x_i,x_j\) lying on the same interior plane, then, \(\lambda_i, \lambda_j>0\),
        \item[(c)] If \(\lambda_i<0 \), then the point \(x_i\) belongs to one of the edges of the unit cube that is not adjacent to  any of \((0,0,0)\) and \((1,1,1)\),
        \item[(d)]  If any point \(x_i\) of \(p\), other than \((0,0,0)\) and \((1,1,1)\), belongs to one of the edges of the unit cube, then \(\lambda_i<0\),
        \item[(e)] Both of the points \((0,0,0)\) and \((1,1,1)\) belong to \(p\) and, both have positive weights.
    \end{enumerate}
\end{theorem}
\begin{proof}

\textbf{(a) }Assume the contrary: Let \((p,\lambda)=((x_0,x_1,...,x_m),(\lambda_0,\lambda_1,...,\lambda_m)) \in [S]_f\) contain a non well-ordered pair of points \(x_k,x_j\) with \(\lambda_k>0,\lambda_j>0\).
    %be an optimal projection cycle-vectorfor \(f\) on the set \(S\) that . 
    
    Let the coordinates of the points \(x_k,x_j\) be \((\xi_k,\gamma_k,\zeta_k)\) and \((\xi_j,\gamma_j,\zeta_j)\), respectively. Since \(x_k,x_j\) are not well-ordered, without loss of generality we can assume that \(\xi_k \ge \xi_j\), \(\gamma_k > \gamma_j\) and \(\zeta_k < \zeta_j\).
    
    Denote the points \((\xi_k,\gamma_k,\zeta_j)\) and \((\xi_j,\gamma_j,\zeta_k)\) by \(x_k'\) and \(x_j'\), respectively.
    Clearly, \(x_k', x_j' \in S\) and the projection cycle-vector \((p_0,\lambda_\varepsilon) \in [S]\) with \(p_0=(x_k,x_j,x_k',x_j')\) and \(\lambda_\varepsilon=(-\varepsilon,-\varepsilon,\varepsilon,\varepsilon)\) is well defined for any \(\varepsilon>0\). Let, \((p',\lambda')=(p,\lambda)+(p_0,\lambda_\varepsilon)\).
    
   Since \(\lambda(x_k)\) and \(\lambda(x_j)\) are positive, when \(\varepsilon>0\) is small enough by Lemma \ref{lemma-denominators} we have \( 0 \le \sum_{}|\lambda'_i| \le  \sum|\lambda_i|\). On the other hand, 
     \[
     \sum_{}\lambda'_if(x_i)-\sum_{}\lambda_if(x_i)=\varepsilon(f(x'_k)+f(x'_j)-f(x_k)-f(x_j))>0,
     \]
where the last inequality follows by Lemma \ref{main-func-ineq} applied to the points \(x_k\) and \(x_j\) and the positivity of \(\varepsilon\).

Now, since \(\sum_{}\lambda'_if(x_i) > \sum_{}\lambda_if(x_i) \ge 0 \) and \( 0<\sum_{}|\lambda'_i| \le \sum_{}|\lambda_i|\), we have
     \[
    \frac{\sum_{}\lambda'_if(x_i)}{\sum_{}|\lambda'_i|}-\frac{\sum_{}\lambda_if(x_i)}{\sum|\lambda_i|} \ge \frac{\sum_{}\lambda'_if(x_i)-\sum_{}\lambda_if(x_i)}{\sum_{}|\lambda_i|}>0.
     \]
This contradicts the optimality of $(p,\lambda)$.

\textbf{(b)} 
    Let \((p,\lambda)=((x_0,x_1,...,x_m),(\lambda_0,\lambda_1,...,\lambda_m))\) be an optimal projection cycle-vector for \(f\) on the set \(S\) that contains a well-ordered pair of points \(x_i,x_j\) from the same interior plane.
    
    Without loss of generality, we can assume that \(x_i\) and \(x_j\) are lying on the same plane parallel to the \(xOy\) plane, and thus we can write \(x_i=(\xi_i,\gamma_i,\zeta)\) and \(x_j=(\xi_j,\gamma_j,\zeta)\) for some numbers \(\xi_i,\xi_j,\gamma_i,\gamma_j \in[0,1]\) and some \(\zeta \in (0,1)\). 
    
    Assume \(\xi_i \le \xi_j\), then, since \(x_i,x_j\) are well-ordered we have \(\gamma_i \le \gamma_j\).
    If both \(\lambda_i,\lambda_j\) are not positive, have three cases:
    \begin{enumerate}
        \item[\textbf{ 1.}] \(\lambda_i>0>\lambda_j\)\\
        Take \(x'_i=(\xi_i,\gamma_i,0)\), \(x'_j=(\xi_j,\gamma_j,0)\), \(p_0=(x_i,x_j,x_i',x_j')\) and \(\lambda_\varepsilon=(-\varepsilon,\varepsilon,\varepsilon,-\varepsilon)\);
        \item[\textbf{ 2.}] \(\lambda_i<0<\lambda_j\)\\
        Take \(x'_i=(\xi_i,\gamma_i,1)\), \(x'_j=(\xi_j,\gamma_j,1)\), \(p_0=(x_i,x_j,x_i',x_j')\) and \(\lambda_\varepsilon=(\varepsilon,-\varepsilon,-\varepsilon,\varepsilon)\);
        \item[\textbf{ 3.}] \(\lambda_i,\lambda_j<0\)\\
        Take \(x'_i=(\xi_i,\gamma_j,\zeta)\), \(x'_j=(\xi_j,\gamma_i,\zeta)\), \(p_0=(x_i,x_j,x_i',x_j')\) and \(\lambda_\varepsilon=(\varepsilon,\varepsilon,-\varepsilon,-\varepsilon)\).
    \end{enumerate}
    \begin{figure}[ht]
\centering
\begin{tikzpicture}[scale=0.30, line join=round, line cap=round]

\draw[->, thick, red] (0,0,13) -- (10,0,13) node[below left] {$x$};
\draw[->, thick, red] (0,0,13) -- (0,10,13) node[left] {$z$};
\draw[->, thick, red] (0,0,13) -- (0,0,-3) node[left] {$y$};

\coordinate (xi') at (2,0,11);
\coordinate (xj') at (7,0,4);
\coordinate (xi) at (2,7,11);
\coordinate (xj) at (7,7,4);

\coordinate (xip) at (7,7,11);
\coordinate (xjp) at (7,0,11);
%planes
\fill[black!30, opacity=0.111] (0,0,0)--(9,0,0)--(9,0,13) -- (0,0,13);
\fill[black!30, opacity=0.204] (0,0,0)--(0,8,0)--(0,8,13) -- (0,0,13);
\fill[black!30, opacity=0.3] (0,0,0)--(0,8,0)--(9,8,0)--(9,0,0);
% Draw the two parallelograms (sides of a prism)
\fill[cyan!30, opacity=0.313] (xi) -- (xi') -- (xjp) -- (xip) -- cycle;
\fill[red!30, opacity=0.313] (xj') -- (xj) -- (xip) -- (xjp) -- cycle;

% Edges

\draw[thick, green!60!black] (xj') -- (xjp);
\draw[thick, green!60!black] (xjp) -- (xip);
\draw[thick, green!60!black] (xjp) -- (xi');

\draw[thick, green!60!black] (xj) -- (xip);
\draw[thick, green!60!black] (xi) -- (xip);

% Vertical edges
\draw[thick, green!60!black] (xi') -- (xi);
\draw[thick, green!60!black] (xj') -- (xj);

% Red dots (middle edges between layers)
\filldraw[red] (xip) circle (2pt);
\filldraw[red] (xjp) circle (2pt);

% Nodes
\filldraw[black] (xi') circle (2pt) node[left] {$x_i'$};
\filldraw[black] (xi) circle (2pt) node[left] {$x_i$};
\filldraw[black] (xj') circle (2pt) node[right] {$x_j'$};
\filldraw[black] (xj) circle (2pt) node[right] {$x_j$};

\end{tikzpicture}
    \caption{\textbf{Case 1:} We simply replace,the pair \(x_i,x_j\) with \(x'_i,x'_j\) }
\end{figure}

    In each case we have chosen \(\lambda_\varepsilon\) in a way that, at both of the points \(x_i\) and \(x_j\) the projection cycle-vectors \((p,\lambda)\) and \((p_0,\lambda_\varepsilon)\) have weights with opposite signs. Thus, by the Lemma \ref{lemma-denominators}, for \(\varepsilon>0\) small enough, we have 
    \[
    0 < \sum_{}|\lambda'_i| \le  \sum|\lambda_i|,
    \]
where \((p',\lambda')=(p,\lambda)+(p_0,\lambda_\varepsilon)\).

On the other hand, applying the Lemma \ref{main-func-ineq} gives
\[
\sum_{x \in p'}\lambda'(x)f(x) \ge \sum_{x \in p}\lambda(x)f(x).
\]
Hence, in all three cases, we will get a contradiction to the optimality of \((p,\lambda)\) simultaneously as in part \((a)\).

\textbf{(c)} Let \((p,\lambda)=(x_0,x_1,...,x_m,\lambda_0,\lambda_1,...,\lambda_m)\) be an optimal projection cycle-vector for \(f\) on the set \(S\) that contains a point \(x_i=(\xi_i,\gamma_i,\zeta_i)\) with \(\lambda_i<0\).

Consider the planes \(x=\xi_i,y=\gamma_i\) and \(z=\zeta_i\). By the definition of the projection cycle, each of these planes contains at least one point of \(p\) with positive weight. Let these points be \(x_j=(\xi_i,\gamma_j,\zeta_j),x_l=(\xi_l,\gamma_i,\zeta_l),x_k=(\xi_k,\gamma_k,\zeta_i)\).
%with weights \(\lambda_j,\lambda_k,\lambda_l>0\) respectively. 
If any of these positive-weighted points makes a well-ordered pair with \(x_i\), then the projection cycle-vector \((p,\lambda)\) cannot be optimal by part \((b)\). 
\begin{figure}[ht]
        \centering
        \begin{tikzpicture}[scale=0.26, line join=round, line cap=round]
% Axis
\draw[->, thick, black] (0,0,-3) -- (15,0,-3) node[below] {$x$};
\draw[->, thick, black] (0,0,-3) -- (0,15,-3) node[left] {$z$};
\draw[->, thick, black] (0,0,-3) -- (0,0,-16) node[left] {$y$};
%planes
\fill[black!30, opacity=0.1] (0,0,-3) -- (14,0,-3) -- (14,0,-12) -- (0,0,-12) -- cycle;
\fill[black!30, opacity=0.3] (0,0,-3) -- (0,12,-3) -- (0,12,-12) -- (0,0,-12) -- cycle;
\fill[black!30, opacity=0.2] (14,0,-12) -- (14,12,-12) -- (0,12,-12) -- (0,0,-12) -- cycle;

\coordinate (xi) at (8,6,-6);

\fill[cyan!30, opacity=0.513] (xi) -- (8,6,-12) -- (8,0,-12) -- (8,0,-6) -- cycle;
\fill[red!30, opacity=0.513] (xi) -- (8,6,-12) -- (0,6,-12) -- (0,6,-6) -- cycle;
\fill[green!30, opacity=0.513] (xi) -- (8,12,-6) -- (0,12,-6) -- (0,6,-6) -- cycle;
\fill[green!30, opacity=0.513] (xi) -- (8,0,-6) -- (14,0,-6) -- (14,6,-6) -- cycle;
\fill[cyan!30, opacity=0.513] (xi) -- (8,6,0) -- (8,12,0) -- (8,12,-6) -- cycle;
\fill[red!30, opacity=0.513] (xi) -- (8,6,0) -- (14,6,0) -- (14,6,-6) -- cycle;
%lines
\draw[thick, black!60!black] (xi) -- (8,6,-12) -- (8,0,-12) -- (8,0,-6) -- cycle;
\draw[thick, black!60!black] (xi) -- (8,6,-12) -- (0,6,-12) -- (0,6,-6) -- cycle;
\draw[thick, black!60!black] (xi) -- (8,12,-6) -- (0,12,-6) -- (0,6,-6) -- cycle;
\draw[thick, black!60!black] (xi) -- (8,0,-6) -- (14,0,-6) -- (14,6,-6) -- cycle;
\draw[thick, black!60!black] (xi) -- (8,6,0) -- (8,12,0) -- (8,12,-6) -- cycle;
\draw[thick, black!60!black] (xi) -- (8,6,0) -- (14,6,0) -- (14,6,-6) -- cycle;

% Nodes
\filldraw[black] (xi) circle (5pt) node[label={[xshift=0.20cm, yshift=0.1cm]{$x_i$}}]{};

\end{tikzpicture}
        \caption{The possible locations of the points that balances \(x_i\)} due to part \((b)\).
        \label{fig:(c)}
\end{figure}
 
Assume both \(0<\xi<1\) and \(0<\gamma<1\). By the part \((b)\) we have that \((x_i,x_j)\) is not a well-ordered pair,thus, either \(\gamma_i< \gamma_j, \zeta_i > \zeta_j\) or \(\gamma_i> \gamma_j, \zeta_i < \zeta_j\). Similarly, \((x_i,x_l)\) is not a well-ordered pair, thus, either \(\xi_i< \xi_l, \zeta_i > \zeta_l\) or \(\xi_i> \xi_l, \zeta_i < \zeta_l\). It is easy to verify that in each of the four possible cases, \(x_j,x_l\) cannot be a well-ordered pair. By part \((a)\) if \((p,\lambda)\) is optimal, then any pair of points \(x_j,x_l,x_k\) should be well-ordered. Thus, \((p,\lambda)\) cannot be optimal, and we get a contradiction.
Since we choose \(\xi,\gamma\) arbitrarily from \(\xi,\gamma,\zeta\), this contradiction works if at least two of the planes \(x=\xi_i,y=\gamma_i\) and \(z=\zeta_i\) are inner planes, which happens when \(x_i\) is not on the edges of the unit cube.
Hence, \(x_i\) belongs to one of the edges of the unit cube. 

Now, if, \(x_i\) and \((0,0,0)\) are on the same edge, we have three cases: \(x_i=(\xi,0,0)\) or \(x_i=(0,\gamma,0)\) or \(x_i=(0,0,\zeta)\). Assume that \(x_i=(\xi,0,0)\), then on the inner plane \(x=\xi\) there exist another point \(x_j=(\xi,\gamma_j,\zeta_j) \in p\) with \(\lambda_j>0\). Moreover, since \(\gamma_j,\zeta_j \ge 0\), \(x_i\) and \(x_j\) are well-ordered. But this contradicts the optimality of \((p,\lambda)\), by part \((b)\). Similarly, in the other two cases, we get a contradiction.

Using similar arguments, it can be proven that \(x_i\) and \((1,1,1)\) cannot be on the same edge.

Thus, \(x_i\) belongs to one of the edges of the cube, which does not contain any of the vertices \((0,0,0)\) and \((1,1,1)\).
Let's name these edges \(l_1,...,l_6\) as following 
\begin{align*}    
    l_1=\{(\xi,0,1), 0 \le \xi \le 1\}, \\
    l_2=\{(1,0,\zeta), 0 \le \zeta \le 1\},\\
    l_3=\{(1,\gamma,0), 0 \le \gamma \le 1\}, \\
    l_4=\{(\xi,1,0), 0 \le \xi \le 1\},\\
    l_5=\{(0,1,\zeta), 0 \le \zeta \le 1\}, \\
    l_6=\{(0,\gamma,1), 0 \le \gamma \le 1\}.\\
\end{align*}

\textbf{(d)}
    Assume the contrary, let \(x_i \in p\) be a point on one of the edges of the unit cube with \(\lambda_i>0\). We have three cases:
    \textbf{Case 1:} \(x_i\) is on the same edge with one of the vertices \((0,0,0)\) or \((1,1,1)\).\\
    If \(x_i\) and \((0,0,0)\) are on the same edge, have three cases: \(x_i=(\xi,0,0)\) or \(x_i=(0,\gamma,0)\) or \(x_i=(0,0,\zeta)\). Assume \(x_i=(\xi,0,0)\), then on the inner plane \(x=\xi\) there exists another point \(x_j=(\xi,\gamma_j,\zeta_j) \in p\) with \(\lambda_j<0\). Moreover, since \(\gamma_j,\zeta_j \ge 0\), \(x_i\) and \(x_j\) are well-ordered. But this contradicts the optimality of \((p,\lambda)\), by part \((b)\). 
    
    Similarly, in the other two cases, we get a contradiction.
    Using similar arguments, it can be proven that \(x_i\) and \((1,1,1)\) cannot be on the same edge.
    
    \textbf{Case 2:} \(x_i\) belongs to one of the edges of the cube, not containing any of the vertices \((0,0,0)\) and \((1,1,1)\). Without loss of generality, we can assume that \(x_i=(\xi,0,1)\). To make the sum of weights on the plane \(z=1\) equal to zero, there should exist a point \(x_j\) of \(p\) with a negative weight on \(z=1\). By part \((c)\) the point \(x_j\) belongs to at least one of the edges \(l_1\) or \(l_6\). Thus, we have two sub-cases:
    
    \textbf{Case 2.1:} \(x_j=(\xi_j,0,1)\) for some \(0 \le \xi_j \le 1\).  Clearly, \(\xi_j \ne \xi\). We should consider two sub-cases:
    
    \textbf{Case 2.1.1:} \(\xi_j < \xi\)\\
    Consider the projection cycle \((p, \lambda_\varepsilon)\), where \(p_0=(x_i, x_j, (\xi,1,1), (\xi_j,1,1))\) and \(\lambda_\varepsilon=(-\varepsilon,\varepsilon,\varepsilon,-\varepsilon)\);
    
    \textbf{Case 2.1.2:} \(\xi < \xi_j\)\\
    Consider the projection cycle \((p_0, \lambda_\varepsilon)\), where \(p_0=(x_i, x_j, (\xi,0,0), (\xi_j,0,0))\) and \(\lambda_\varepsilon=(\varepsilon,-\varepsilon,-\varepsilon,\varepsilon)\).

    Similarly to part \((a)\) and part \((b)\), taking \((p',\lambda')=(p,\lambda)+(p_0,\lambda_\varepsilon)\) we get a contradiction to the optimality of \((p,\lambda)\).
    
    \textbf{Case 2.2:} \(x_j=(0,\gamma,1)\) for some \(0 < \gamma \le 1\). 
    Consider the plane, \(y=\gamma\). To make the sum of weights on the plane \(y=\gamma\) equal to zero, there should exist a point \(x_k\) of \(p\) with a positive weight on \(y=\gamma\). Let \(x_k=(\xi_k,\gamma,\zeta)\). If \(\zeta=1\), the points \(x_j\) and \(x_k\) belonging to the inner plane \(y=\gamma\) becomes well-ordered and since \(x_j\) and \(x_k\) have opposite weights we get a contradiction by part \((b)\). If, \(\zeta<1\), the points, \(x_i\) and \(x_k\) becomes non well-ordered. Since both of \(x_i\) and \(x_k\) have positive weights, we get a contradiction by the part \((a)\).
    \\
    Since we get a contradiction in each case, the proof of part \((d)\) is completed.

\textbf{(e)} 
    Consider the point \(x_i=(\xi_i,\gamma_i,\zeta_i)\in p\) with minimal coordinates among the positive weighted points of \(p\). (Since the positive weighted points are well-ordered pairwise, the point with the minimal sum of coordinates has minimal coordinates.) 
    
    If \(x_i=(0,0,0)\), we are done. Otherwise, since \(x_i\) is chosen to be minimal, it cannot be any other vertex of the unit cube, and thus there is at least one inner plane passing through \(x_i\). 
    Without loss of generality, assume that this plane has equation \(x=\xi_i\). Then, on the plane \(x=\xi_i\) there is another point \(x_j=(\xi_i,\gamma_j,\zeta_j) \in p\) such that \(\lambda_j<0\). By part \((b)\) we know that \(x_i\) and \(x_j\) cannot be well-ordered, and hence, either \(\gamma_j<\gamma_i\) or \(\zeta_j<\zeta_i\) but in the first case on the plane \(y=\gamma_j\) and in the second case on the plane \(z=\zeta_j\) there is a positive weighted point to make the sum over that plane equal to zero. But this contradicts to the minimality of \(x_i\) and, hence, \(x_i=(0,0,0)\).

    Similarly, we can prove that the positive weighted point with maximal coordinates is \((1,1,1)\).

%\end{enumerate}
\begin{figure}[ht]
\centering
\begin{tikzpicture}[scale=0.3, line join=round, line cap=round]
% Axis
\draw[->, red] (0,0,9) -- (12,0,9) node[below left] {$x$};
\draw[->, red] (0,0,9) -- (0,12,9) node[left] {$z$};
\draw[->, red] (0,0,9) -- (0,0,-5) node[left] {$y$};

\coordinate (T_4) at (9,8,0);
\coordinate (T_3) at (0,8,0);
\coordinate (T_2) at (9,0,0);
\coordinate (T_1) at (0,0,0);
\coordinate (T_5) at (9,8,9);
\coordinate (T_6) at (0,8,9);
\coordinate (T_7) at (9,0,9);
\coordinate (T_8) at (0,0,9);
\coordinate (T_x) at (4,4,4);
\coordinate (T_y) at (2,3,3);
\coordinate (T_z) at (7,6,7);

% intersection (z-direction)
\coordinate (xip) at (7,7,11);
\coordinate (xjp) at (7,3,11);

\draw[thick, green!60!black] (T_1) -- (T_2);
\draw[thick, green!60!black] (T_1) -- (T_3);
\draw[thick, green!60!black] (T_3) -- (T_6);
\draw[thick, green!60!black] (T_6) -- (T_5);
\draw[thick, green!60!black] (T_5) -- (T_7);
\draw[thick, green!60!black] (T_7) -- (T_2);
\draw[red] (T_4) -- (T_3);
\draw[red] (T_4) -- (T_5);
\draw[red] (T_4) -- (T_2);

\filldraw[red] (T_1) circle (5pt) node[right] {};
\filldraw[red] (T_2) circle (5pt) node[right] {};
\filldraw[red] (T_3) circle (5pt) node[right] {};
\filldraw[black] (T_4) circle (5pt) node[right] {};
\filldraw[red] (T_5) circle (5pt) node[right] {};
\filldraw[red] (T_6) circle (5pt) node[right] {};
\filldraw[red] (T_7) circle (5pt) node[right] {};
\filldraw[black] (T_8) circle (5pt) node[right] {};
\filldraw[black] (T_x) circle (5pt) node[right] {};
\filldraw[blue] (9,0,4) circle (3pt) node[right] {};
\filldraw[blue] (4,0,0) circle (3pt) node[right] {};
\filldraw[blue] (0,8,4) circle (3pt) node[right] {};
\filldraw[blue] (0,4,0) circle (3pt) node[right] {};
\filldraw[blue] (4,8,9) circle (3pt) node[right] {};
\filldraw[blue] (9,4,9) circle (3pt) node[right] {};
%\filldraw[black] (T_y) circle (5pt) node[right] {};
%\filldraw[black] (T_z) circle (5pt) node[right] {};
\fill[black!30, opacity=0.111] (4,0,0)--(4,8,0)--(4,8,9)--(4,0,9);
\fill[black!30, opacity=0.211] (9,0,4)--(9,8,4)--(0,8,4)--(0,0,4);
\fill[black!30, opacity=0.11] (0,4,0)--(0,4,9)--(9,4,9)--(9,4,0);
\end{tikzpicture}
         \label{f:Lemma 2.2}
         \caption{Geometric interpretation: Theorem 2.2}
\end{figure}

\end{proof} 
From now on, \([S]^{(a)-(e)}\) will denote the set of projection cycle-vectors \((p,\lambda) \in [S]\), such that the propositions \((a)-(e)\) of the Theorem 2.2 holds for \((p,\lambda)\) and let \([S]^{(a)-(e)}_{min}\) denote the set of all minimal projection cycle-vectors \((p,\lambda) \in [S]^{(a)-(e)}\).

The following lemmas will later help us to prove our main results for functions satisfying the inequalities \eqref{e: delta class definition}, instead of the stricter version \eqref{e:strict-class-definition}.
\begin{lemma}
    The vector assigned to a minimal projection cycle is unique up to multiplication by a constant. Moreover, if it is orthonormal, then it has rational components.
    Thus, any minimal projection cycle uniquely (up to a sign) defines an orthonormal projection cycle-vector.
\end{lemma}
\begin{proof}
    (See: \cite[Lemma 4 and Lemma 5]{Navada} or \cite[Lemma 3.32]{Ismailov_2021}).
\end{proof}
\begin{lemma}
    For any projection cycle \((p,\lambda) \in [S]\) and for any function \(f\in C(\Omega)\), there exist a projection cycle \((p^0,\lambda^0) \in [S]_{min}\) such that \(p^0 \subset p\) and the following relations hold
    \begin{enumerate}
        \item \[\lambda^0(x)\lambda(x)>0, \ \forall x \in p^0, \]
        \item \[\frac{\sum_{x_j \in p^0}\lambda^0_jf(x_j)}{\sum|\lambda^0_j|} \ge \frac{\sum_{x_i \in p}\lambda_if(x_i)}{\sum|\lambda_i|}.\]
    \end{enumerate}
\end{lemma}
\begin{proof}
    See \cite[Theorem 2]{Navada}.
\end{proof}
Lemma 2.4 implies that, for any function \(f\),
\[
\max_{(p,\lambda)\in [S]} \frac{\sum_{}\lambda_if(x_i)}{\sum|\lambda_i|}=\max_{(p,\lambda)\in [S]_{min}} \frac{\sum_{}\lambda_if(x_i)}{\sum|\lambda_i|}
\]
and
\[
\max_{(p,\lambda)\in [S]^{(a)-(e)}} \frac{\sum_{}\lambda_if(x_i)}{\sum|\lambda_i|}=\max_{(p,\lambda)\in [S]^{(a)-(e)}_{min}} \frac{\sum_{}\lambda_if(x_i)}{\sum|\lambda_i|}.
\]
\begin{lemma} 
    For any set \(S=S_x\times S_y\times S_z\), where \(S_x,S_y,S_z \subset [0,1]\) are some finite sets with \(\{0,1\}\subset S_x \cap S_y \cap S_z\) and for any function \(f \in C(\Omega)\) satisfying the inequalities \eqref{e: delta class definition}, the intersection \([S]^{(a)-(e)} \cap [S]_f\) is non-empty. 
\end{lemma}
\begin{proof}
Let's define 
\[
f_n(x,y,z)=f(x,y,z)+\frac{1}{n}(xy+xz+yz), \ \forall n \in \mathbb{N}.
\]
Clearly, the function \(f_n\) satisfies the inequalities \eqref{e:strict-class-definition} and therefore it follows from Theorem 2.2 and Lemma 2.4 that,
\[
 \max_{(p,\lambda)\in [S]} \frac{\sum_{}\lambda_if(x_i)}{\sum|\lambda_i|} =\max_{(p,\lambda)\in [S]_{min}} \frac{\sum_{}\lambda_if(x_i)}{\sum|\lambda_i|} = \max_{(p,\lambda)\in [S]_{min}} \frac{\sum_{}\lambda_i(\lim_{n \rightarrow  \infty}f_{n}(x_i))}{\sum|\lambda_i|}.
\]
%And since there are only finitely many terms on the numerator, we can take the limit out:
Here, we can change the order of the limit and the maximum, as \([S]_{min}\) is a finite set by Lemma 2.3 (considering vectors up to multiplication by a constant), and hence,
\[
\max_{(p,\lambda)\in [S]_{min}} \frac{\sum_{}\lambda_i(\lim_{n \rightarrow  \infty}f_{n}(x_i))}{\sum|\lambda_i|}=\lim_{n \rightarrow \infty}\max_{(p,\lambda)\in [S]_{min}} \frac{\sum_{}\lambda_if_{n}(x_i)}{\sum|\lambda_i|}=\lim_{n \rightarrow  \infty}\max_{(p,\lambda) \in [S]}\frac{\sum_{}\lambda_if_{n}(x_i)}{\sum|\lambda_i|}.
\]

By the Theorem 2.2 and Lemma 2.4,
\[
\lim_{n \rightarrow  \infty}\max_{(p,\lambda) \in [S]}\frac{\sum_{}\lambda_if_{n}(x_i)}{\sum|\lambda_i|}=\lim_{n \to  \infty}\max_{(p,\lambda) \in [S]^{(a)-(e)}}\frac{\sum_{}\lambda_if_{n}(x_i)}{\sum|\lambda_i|}=\lim_{n \to  \infty}\max_{(p,\lambda) \in [S]^{(a)-(e)}_{min}}\frac{\sum_{}\lambda_if_{n}(x_i)}{\sum|\lambda_i|}.
\]
Changing the order of the limit and maximum once more and using the Lemma 2.3  gives the desired result:
\[
\lim_{n \rightarrow  \infty}\max_{(p,\lambda) \in [S]^{(a)-(e)}_{min}}\frac{\sum_{}\lambda_if_{n}(x_i)}{\sum|\lambda_i|}=\max_{(p,\lambda) \in [S]^{(a)-(e)}_{min}}\frac{\sum_{}\lambda_i(\lim_{n \rightarrow  \infty}f_{n}(x_i))}{\sum|\lambda_i|}=\max_{(p,\lambda) \in [S]^{(a)-(e)}}\frac{\sum_{}\lambda_i(f(x_i))}{\sum|\lambda_i|}.
\]
Now let's choose \((p^0,\lambda^0) \in [S]^{(a)-(e)}\) such that,
\[
\frac{\sum_{}\lambda^0_i(f(x_i))}{\sum|\lambda^0_i|}=\max_{(p,\lambda) \in [S]^{(a)-(e)}}\frac{\sum_{}\lambda_i(f(x_i))}{\sum|\lambda_i|}.
\]
Then, 
\[
\frac{\sum_{}\lambda^0_i(f(x_i))}{\sum|\lambda^0_i|}=\max_{(p,\lambda) \in [S]}\frac{\sum_{}\lambda_i(f(x_i))}{\sum|\lambda_i|},
\]
and thus, \((p^0,\lambda^0) \in [S]_f\) by the definition of \([S]_f\). Hence, \((p^0,\lambda^0) \in [S]^{(a)-(e)} \cap [S]_f\).
\end{proof}

\subsection{Construction of a finite set containing an extremal projection cycle-vector}
\vspace{15pt}

Above, we study some structural properties of optimal projection cycles on certain finite sets. Below, we will use these properties to prove that there is a projection cycle on an explicitly given finite set that gives the supremum \eqref{e:Golomb-applied}. 

Let \(f:\Omega\rightarrow \mathbb{R}\) be a given continuous function of three variables satisfying conditions (2.1). 

We start by defining the desired set. Let's name the corners of the unit cube from \(T_1\) to \(T_8\) as 
\begin{align*}
    T_1=(0,0,0), \ T_2=(0,0,1), \ T_3=(1,0,1), \ T_4=(1,0,0),\\
    T_5=(1,1,0), \ T_6=(0,1,0), \ T_7=(0,1,1), \ T_8=(1,1,1).
\end{align*}
Define the functions,
\begin{align*}
    g_1(x,y,z)=f(x,y,z)-f(x,0,1)-f(1,y,0)-f(0,1,z), \\
    g_2(x,y,z)=f(x,y,z)-f(1,0,z)-f(x,1,0)-f(0,y,1), \\
    g_3(x,y,z)=f(x,y,z)-f(x,0,1)-f(1,0,z)-f(1,y,0), \\
    g_4(x,y,z)=f(x,y,z)-f(1,0,z)-f(1,y,0)-f(x,1,0), \\
    g_5(x,y,z)=f(x,y,z)-f(1,y,0)-f(x,1,0)-f(0,1,z), \\
    g_6(x,y,z)=f(x,y,z)-f(x,1,0)-f(0,1,z)-f(0,y,1), \\
    g_7(x,y,z)=f(x,y,z)-f(0,1,z)-f(0,y,1)-f(x,0,1), \\
    g_8(x,y,z)=f(x,y,z)-f(0,y,1)-f(x,0,1)-f(1,0,z)
\end{align*}
on the unit cube and
\begin{align*}
    h_1(x,y,0)=f(x,y,0)-f(0,y,1)-f(x,0,1),\\
    h_2(x,y,1)=f(x,y,1)-f(1,y,0)-f(x,1,0),\\
    h_3(x,0,z)=f(x,0,z)-f(0,1,z)-f(x,1,0),\\
    h_4(x,1,z)=f(x,1,z)-f(1,0,z)-f(x,0,1),\\
    h_5(0,y,z)=f(0,y,z)-f(1,0,z)-f(1,y,0),\\
    h_6(1,y,z)=f(1,y,z)-f(0,1,z)-f(0,y,1)
\end{align*}
on the faces of the unit cube lying on the planes \(z=0,z=1,y=0,y=1,x=0\) and \(x=1\) correspondingly.

Let \(M_i,1\le i \le 8\) be a maximum point of the function \(g_i\) on the domain \(\Omega=[0,1]^3\).
And \(F_i,1\le i \le 6\) be a maximum point of the function \(h_i\) on the face it is defined.

Let \[U=U(f)=\{T_i, 1 \le i \le 8\} \cup\{M_i, 1 \le i \le 8\} \cup \{F_i, 1 \le i \le 6\},\]
and let \(U_x,U_y,U_z\) be the set of projections of elements of \(U\) to the coordinate lines \(Ox,Oy,Oz\) respectively. 
%We will call \(U\) the set of maximal efficient points. 
We will see that the choice of \(M_i\) and \(F_i\) among the maximum points of \(g_i\) and \(h_i\) will not change the correctness of the following theorem.

\begin{remark}
    Some of the points $T_i, \ 1 \le i \le 8$, $M_j, \ 1 \le j \le 8$, $F_k,\ 1 \le k \le 6$ may coincide. In this case, the set $U$ will narrow.
\end{remark}

\begin{theorem} \label{t:Theorem 1}
    Let \(f \in C(\Omega) \) be a function satisfying the inequalities \eqref{e: delta class definition} and  \(S =S_x \times S_y \times S_z \subset \Omega \) be any arbitrary finite set with \(U \subset S\), then we have:
    \begin{equation} \label{e:maximum-clarified}
        \max_{(p,\lambda)\in [S] }\frac{\sum_{}\lambda_if(x_i)}{\sum|\lambda_i|} =\max_{(p,\lambda) \in [U_x\times U_y \times U_z]}\frac{\sum_{} \lambda_if(x_i)}{\sum|\lambda_i|}.
    \end{equation}
\end{theorem}
\begin{proof}
Let, \([S]^{(a)-(d)}\) be the set of projection cycle-vectors \((p,\lambda) \in [S]\) that satisfies propositions \((a),(b),(c),(d)\) of the Theorem 2.2.

It is clear that \([S]^{(a)-(e)} \subset [S]^{(a)-(d)}\), and hence from Lemma 2.5 we have
\[
\max_{(p,\lambda)\in [S] }\frac{\sum_{}\lambda_if(x_i)}{\sum|\lambda_i|}=\max_{(p,\lambda)\in [S]^{(a)-(d)} }\frac{\sum_{}\lambda_if(x_i)}{\sum|\lambda_i|}.
\]
Thus, it is enough to prove that
\[
\max_{(p,\lambda)\in [S]^{(a)-(d)} }\frac{\sum_{}\lambda_if(x_i)}{\sum|\lambda_i|}=\max_{(p,\lambda) \in [U_x\times U_y \times U_z]}\frac{\sum_{} \lambda_if(x_i)}{\sum|\lambda_i|}.
\]
Since \(U \subset S\) and \(S=S_x \times S_y \times S_z\) we have \(U_x \subset S_x\), \(U_y \subset S_y\) and \(U_z \subset S_z\) which imply \(U_x\times U_y \times U_z \subset S\) and therefore we can write the inequality
\[
\max_{(p,\lambda)\in [S]^{(a)-(d)} }\frac{\sum_{}\lambda_if(x_i)}{\sum|\lambda_i|} \ge \max_{(p,\lambda) \in [U_x\times U_y \times U_z]}\frac{\sum_{} \lambda_if(x_i)}{\sum|\lambda_i|}.
\]
Assume the contrary to the theorem, let there exists a projection cycle \((p^0,\lambda^0) \in [S]^{(a)-(d)}\) such that
\begin{equation} \label{e: contrary assumption}
    \frac{\sum_{}\lambda^0_if(x^0_i)}{\sum|\lambda^0_i|} > \max_{(p,\lambda) \in [U_x\times U_y \times U_z]}\frac{\sum_{} \lambda_if(x_i)}{\sum|\lambda_i|}. 
\end{equation}
 
Clearly, \((p^0,\lambda^0) \not \in [U_x\times U_y \times U_z]\). Moreover we can choose it to be optimal in the set \(S\) for the  function \(f\), that is, \((p^0,\lambda^0) \in [S]_f\).

If all points in \(p^0\) with positive weight belong to \(U\), then the set of their projections on the edges of the unit cube (which are all the possible locations of points of \(p^0\) with negative weights)  would belong to the set \(U_x \times U_y \times U_z\), and hence we would have \(p^0 \subset U_x \times U_y \times U_z\).

Thus, there should exists at least a point \(x \in p^0\) with positive weight(\(\lambda^0(x)>0\)) such that, \(x \not \in U\). We will move weights from each such point to a point of the set \(U\) by help of the Lemma 2.2. By Theorem 2.2, \((d)\), we know that \(x\) cannot be on the edges of the unit cube. 

We define our moves based on the location of \(x\) in the unit cube. Thus, we have the following cases:
\item[\textbf{Case 1:}] \(x\) is an interior point of one of the faces of the unit cube \(\Omega\).

Without loss of generality, we can assume that \(x\) belongs to the face lying on the \(xOy\) plane. Then we can write \(x=(\xi,\gamma,0)\) for some \(\xi,\gamma \in (0,1)\). 
    
Consider the planes \(x=\xi\) and \(y=\gamma \). In each of these planes, there should be at least one point \((p^0,\lambda^0)\) with negative weight. By Theorem 2.2, part \((c)\), these negative weighted points could only be on the edges, not adjacent to any of the vertices \((0,0,0)\) and \((1,1,1)\). The intersections of the planes \(x=\xi\) and \(y=\gamma \) and the edges are the points \((\xi,1,0),(\xi,0,1)\) and \((1,\gamma,0),(0,\gamma,1)\), respectively. 

Since both \((\xi,1,0)\) and \((1,\gamma,0)\) are well-ordered with \(x\) and lie on an interior plane, by Theorem 2.2, part \((b)\), none of \((\xi,1,0)\) and \((1,\gamma,0)\) can have a negative weight. Thus, both \((\xi,0,1)\) and \((0,\gamma,1)\) have negative weights. Moreover, since no other negative weighted points can exist on the planes \(x=\xi\) and \(y=\gamma \), we have:
\begin{equation*} \label{e:inequalities for weights}
|\lambda^0(\xi,0,1)|\ge \lambda^0(x)\text{ and } |\lambda^0(0,\gamma,1)| \ge \lambda^0(x)
\end{equation*}
and thus \(\min(|\lambda^0(\xi,0,1)|,|\lambda^0(0,\gamma,1)|,|\lambda^0(x)|)=\lambda^0(x)\).
     
Let's move the weights on the points \(x\),\((\xi,0,1)\) and \((0,\gamma,1)\) to the points \(F_1=(\xi_1,\gamma_1,0)\) and its two corresponding projections, namely \((\xi_1,0,1)\) and \((0,\gamma_1,1)\). Or, in other words, we define \[p'=(x,(\xi,0,1)\,(0,\gamma,1),(\xi_1,\gamma_1,0),(\xi_1,0,1),(0,\gamma_1,1))\]
and vector 
    \[ \lambda_{\varepsilon}=(-\varepsilon,\varepsilon,\varepsilon,\varepsilon,-\varepsilon,-\varepsilon), \]
    where \(\varepsilon=\lambda^0(x)\).
    Let \((p^1,\lambda^1)=(p^0,\lambda^0)+(p',\lambda_{\varepsilon})\). Since, at each of the first three points of \(p'\), the weights defined in \((p',\lambda_{\varepsilon})\) and \((p^0,\lambda^0)\) have opposite signs and
    \[ \varepsilon=\min(|\lambda^0(\xi,0,1)|,|\lambda^0(0,\gamma,1)|,|\lambda^0(x))|, \]
    by Lemma 2.2 we have
    %for arbitrarily small \(\varepsilon>0\) we get
    \[
    \sum|\lambda^1_i| \le \sum|\lambda^0_i|.
    \]
On the other hand, 
\[
    \sum\lambda^1_if(x_i)-\sum\lambda^0_if(x_i)=
\]
\[
    =f(\xi_1,\gamma_1,0)-f(\xi_1,1,0)-f(1,\gamma_1,0)-f(\xi,\gamma,0)+f(\xi,1,0)+f(1,\gamma,0)=h(\xi_1,\gamma_1,0)-h(\xi,\gamma,0) \ge 0,
\]
where the inequality follows from the choice of \(F_1=(\xi_1,\gamma_1,0)\) as the maximum of the function \(h_1\).
Thus, we get
\[
    \frac{\sum\lambda^1_if(x_i)}{\sum|\lambda^1_i|} \ge \frac{\sum\lambda^0_if(x_i)}{\sum|\lambda^0_i|} >\max_{(p,\lambda) \in [U_x\times U_y \times U_z]}\frac{\sum_{} \lambda_if(x_i)}{\sum|\lambda_i|}.
\]
This proves that, \((p^1,\lambda^1) \in [S]_{f}\).
%Clearly, \((p^1,\lambda^1) \in [S]^{(a)-(d)}\) as no negative weighted point became positive weighted and vice versa. %Moreover, \(\lambda^1(x)=\lambda^0(x)-\lambda_\varepsilon(x)=0\) and hence \((p^1,\lambda^1)\) have fewer positive weighted points not belonging to \(U\) than \((p^0,\lambda^0)\) has.
%Thus, we get a contradiction to the choice of \((p_0,\lambda_0)\).
This makes the weight at point \(x\) equal to zero in the new projection cycle. A similar movement of the weights can be constructed if \(x\) lies on one of the other faces too.\\
\textbf{Case 2:} \(x\) is an interior point of the unit cube \(\Omega\).
    
Let \(x=(\xi, \gamma, \zeta)\). Consider the planes \(x=\xi\), \(y=\gamma\), \(z=\zeta\). Their intersections with the edges of the unit cube that are not adjacent to the vertices not adjacent to any of the vertices \((0,0,0)\) and \((1,1,1)\) will be the points \((\xi,1,0),(\xi,0,1)\) and \((1, \gamma, 0),(0, \gamma, 1)\) and \((1, 0, \zeta),(0, 1, \zeta)\), respectively. To make the sum over the corresponding plane zero, at least one (maybe both) of two points in each plane should belong to \((p^0,\lambda^0)\) with a negative weight.
    
Hence, all points of at least one of the eight sets 
\begin{align*}
    G_1(x)=\{(\xi,0,1),(1,\gamma,0),(0,1,\zeta)\}, \\
    G_2(x)=\{(1,0,\zeta),(\xi,1,0),(0,\gamma,1)\}, \\
    G_3(x)=\{(\xi,0,1),(1,0,\zeta),(1,\gamma,0)\}, \\
    G_4(x)=\{(1,0,\zeta),(1,\gamma,0),(\xi,1,0)\}, \\
    G_5(x)=\{(1,\gamma,0),(\xi,1,0),(0,1,\zeta)\}, \\
    G_6(x)=\{(\xi,1,0),(0,1,\zeta),(0,\gamma,1)\}, \\
    G_7(x)=\{(0,1,\zeta),(0,\gamma,1),(\xi,0,1)\},\\
    G_8(x)=\{(0,\gamma,1),(\xi,0,1),(1,0,\zeta)\}
\end{align*}
belong to \((p^0,\lambda^0)\) and have negative weights.
    
For each \(G_i\) we proceed by moving the weights at point \(x\) and elements of \(G_i\) to \(M_i\) and its corresponding projections on the edges. For example, if all elements of \(G_1\) belong to \(p^0\), then we construct
\(p'=(x,(\xi,0,1),(1,\gamma,0),(0,1,\zeta),(\xi_1,\gamma_1,\zeta_1),(\xi_1,0,1),(1,\gamma_1,0),(0,1,\zeta_1))\)
and vector
\(
\lambda_\varepsilon=(-\varepsilon,\varepsilon,\varepsilon,\varepsilon,\varepsilon,-\varepsilon,-\varepsilon,-\varepsilon)
\)
where \(M_1=(\xi_1,\gamma_1,\zeta_1)\) is the maximum point of \(g_1\) defined above and \(\varepsilon>0\). \\ As in \textbf{Case 1}, let \((p^1,\lambda^1)=(p^0,\lambda^0)+(p',\lambda_\varepsilon).\) Since, at each of the first four points of \((p',\lambda_\varepsilon)\), the weights defined in \((p',\lambda_\varepsilon)\) and \((p^0,\lambda^0)\) have opposite signs, by Lemma 2.2 for \[\varepsilon=\min(|\lambda(x)|,|\lambda((\xi,0,1))|,|\lambda((1,\gamma,0))|,|\lambda((0,1,\zeta))|\] we have 
\[
\sum|\lambda_i^1| \le \sum |\lambda_j^0|.
\]
Moreover, \[
\sum \lambda^1_if(x^1_i)-\lambda^0_if(x^0_i)=g_1(M_1)-g_1(x) \ge 0.\]\\
Thus, we have
\[
\frac{\sum\lambda^1_if(x_i)}{\sum|\lambda^1_i|} \ge \frac{\sum\lambda^0_if(x_i)}{\sum|\lambda^0_i|}.
\]
Again, this proves that, \((p^1,\lambda^1) \in [S]_{f}\).
On the other hand, choice of \(\varepsilon\) makes one of the weights of the points \(x,(\xi,0,1),(1,\gamma,0),(0,1,\zeta)\) to have zero weight in \((p^1,\lambda^1)\). Hence, after a finite number of copies of the above change for some sets \(G_i\), the weight in \(x\)  will be equal to zero in the resulting projection cycle-vector. 
%As in the first case, the last contradicts our assumption of \((p^0,\lambda^0)\) having minimal positive weighted points not belonging to \(U\).
%Since all cases lead to the contradiction, the theorem is proved.

Thus, we can proceed a finite number of described moves and obtain a projection cycle-vector from \([U_x \times U_y \times U_z]\), which is optimal in \(S\).
\end{proof}

The proof implies that the choice of \(M_i\) and \(F_i\) among the maximum points of \(g_i\) and \(h_i\) does not affect the correctness of the theorem. Now, we are ready to state our main result.
\begin{theorem}
    Let \(f \in C(\Omega)\) be a function satisfying the relations \eqref{e: delta class definition}. Then
    \begin{equation}
        E(f,\Omega)=\max_{(p,\lambda) \in [U_x\times U_y \times U_z]^{(a)-(e)}_{min}}\frac{\sum_{} \lambda_if(x_i)}{\sum|\lambda_i|},
    \end{equation}
    where \(U_x, U_y, U_z\) are the projections of \(U\) to the coordinate lines \(Ox,Oy,Oz\) respectively.
\end{theorem}
\begin{proof}
Applying Theorem 2.2 to the relation \eqref{e: supremum-maximum} we get
\begin{equation}
\begin{split}\label{e: main result}
    E(f,\Omega)=\sup_{A,B,C}  \left({\max_{(p,\lambda)\in [A\times B\times C]} \frac{\sum_{}\lambda_if(x_i)}{\sum|\lambda_i|}}\right)=\\=\sup_{ A,B,C}  \left({\max_{(p,\lambda)\in [(U_x \cup A)\times (U_y \cup B) \times (U_z \cup C)]} \frac{\sum_{}\lambda_if(x_i)}{\sum|\lambda_i|}}\right)=\\
   =\sup_{A,B,C}  \left(\max_{(p,\lambda) \in [U_x\times U_y \times U_z]}\frac{\sum_{} \lambda_if(x_i)}{\sum|\lambda_i|}\right)=
   \max_{(p,\lambda) \in [U_x\times U_y \times U_z]}\frac{\sum_{} \lambda_if(x_i)}{\sum|\lambda_i|}.
\end{split}
\end{equation} 
Hence, we are finished by the Lemma 2.4 and Lemma 2.5.
\end{proof}
\begin{remark}
    The above-mentioned papers \cite{Diliberto_Straus},\cite{Flatto}, \cite{Ismailov_2021} give similar theorems. Their set corresponding to \(U\) in our case, was the corners of the unit square in papers  \cite{Diliberto_Straus} and \cite{Flatto}, and points from a predefined grid in \cite{Ismailov_2021}. While all these are independent of \(f\), \(U\) depends on the function \(f\) in the above theorem.  
\end{remark}
\begin{remark}
    The Theorem 2.3 shows that there exists a minimal projection cycle-vector \((p,\lambda) \in [U_x\times U_y \times U_z]\) such that \((p,\lambda)\) satisfies all the propositions of Theorem 2.2 and 
    \[
    E(f,\Omega)=\frac{\sum_{x_i \in p} \lambda(x_i)f(x_i)}{\sum_{x_i \in p}|\lambda(x_i)|}.
    \]
    In the Appendix, we have shown that there are at most 123 such projection cycle-vectors in the set  \(U\). Moreover, any of them can be reconstructed only by knowing the weights in the vertices and all the positive weighted points(if two of those points coincide, then thinking them as separate and adding their weights after the weight calculations are done) by using the definitions of $M_i$ and $F_i$. For example, if $M_i$ has weight $\lambda$, we will add $-\lambda$ to all of the points on $G_i$. 
\end{remark}

\section{Minimal Projection Cycles on a Finite Set}
\vspace{15pt}

Let \(p=(x_1,x_2,...,x_k)\) be an ordered finite subset of \(\Omega\) with \(k\) points. In order to find all vectors \(\lambda=(\lambda_1,\lambda_2,\dots, \lambda_k)\), such that \((p,\lambda)\) is a projection cycle-vector, we have to solve the system of linear equations \eqref{e:projection}. Let this system be represented as 
\begin{equation} \label{e: matrix form}
    A\lambda=0
\end{equation}
for a matrix \(A\).
Below, we will prove some necessary conditions for an ordered subset \(p'\) of \(p\) to be a minimal projection cycle.

We start with a well-known algebraic definition and a lemma (see: \cite[Definition 1 and Proposition 7]{Szalkai_Dosa}). 
\begin{definition}
    Let for any vector \(v=(v_1,...,v_n) \in \mathbb{R}^n\) we define,
    \[
    \operatorname{supp}(v):=\left\{i \leq m: v_i \neq 0\right\}
    \] and call it the support of \(v\). A solution \(v_0\) of the homogeneous linear equation 
    \(
    Mv=0
    \)
    is called a minimal solution if the set \(\operatorname{supp}(v_0)\) is minimal, that is, there exists no other solution \(v\) of the equation satisfying \[\operatorname{supp}(v)\varsubsetneqq \operatorname{supp}(v_0).\]
\end{definition}
\begin{lemma} Let \(M\) be a matrix then a solution \(v_0\) of the equation \(
    Mv=0,
    \)
    which \(i_1,i_2,...,i_t\) -th components are nonzero and all other components are zero, is minimal, then the column system consisting of \(i_1,i_2,...,i_k\) -th columns of the matrix \(M\) is a minimal linearly dependent column system.
\end{lemma}
The next lemma connects minimal solutions and minimal projection cycles.
\begin{lemma}
Let \(p'=(x_{i_1},x_{i_2},\dots x_{i_t}) \subset p\) be a minimal projection cycle, and \(\lambda'=(\lambda'_{i_1},\lambda'_{i_2},\dots \lambda'_{i_t})\) be a vector assigned to \(p'\). Then the vector, \(\lambda=(\lambda_1,...,\lambda_k)\) is a minimal solution of the equation \eqref{e: matrix form}, where
\begin{equation} \label{e: vector-g}
    \lambda_i=
    \begin{cases}
        \lambda'_i, \hspace{13pt} \text{ if } i \in \{i_1,i_2,\dots, i_t\} \\
        0, \hspace{17pt}\text{ otherwise.}
    \end{cases}
\end{equation}

\end{lemma}
\begin{proof}
    Assume the contrary; suppose that \(\lambda\) is not a minimal solution of \eqref{e: matrix form}. Then, there exists a solution \(\lambda^0\) of \eqref{e: matrix form} such that
    \[
    \operatorname{supp}(\lambda^0)\varsubsetneqq \operatorname{supp}(\lambda).
    \]
    Thus, there exists a nonempty set \(\{j_1,j_2,...,j_s\} \varsubsetneqq \{1,2,...,t\}\) such that the \(i_{j_1},i_{j_2},...,i_{j_s}\)-th components of \(\lambda^0\) are nonzero and all other components are zero. And, this implies that \[p''=(x_{i_{j_1}},x_{i_{j_2}},...,x_{i_{j_s}}) \varsubsetneqq p' \] is a projection cycle with vector \(\lambda''=(\lambda^0_{i_{j_1}},\lambda^0_{i_{j_2}},...,\lambda^0_{i_{j_s}})\), which contradicts the minimality of \(p'\).
\end{proof}
The next theorem gives a necessary condition on the minimality of a projection cycle.

\begin{theorem}
    Let, \(a_i\) denote the \(i\)-th column of the matrix \(A\) and \(p'=(x_{i_1},x_{i_2},\dots x_{i_t}) \subset p\) be a minimal projection cycle, then the set of columns \(\{a_{i_1},a_{i_2},\dots a_{i_t}\}\) is a minimal linearly dependent set of columns.
\end{theorem}
\begin{proof}
    Let \(\lambda'=(\lambda'_{i_1},\lambda'_{i_2},\dots \lambda'_{i_t})\) be a vector assigned to \(p'\) and let the vector \(\lambda\) be defined as in \eqref{e: vector-g}. Then, according to Lemma 3.2, \(\lambda\) is a minimal solution of \eqref{e: matrix form}. Thus, by Lemma 3.1 the proof follows.
\end{proof}

\newpage
\section*{Appendix: Finding all minimal projection cycles}
The relation \eqref{e: main result} implies that the error can be found by calculating a maximum over a finite set of the form \(V=U_x \times U_y \times U_z\). Below, we will search for \([V]^{(a)-(e)}_{min}\)-the set of all minimal projection cycle-vectors on set \(V\) satisfying all the propositions from the Theorem  \ref{t: Theorem 2.2}.
 
We can construct a system of equations with variables being the weights at the points of \(U=\{T_i, 1 \le i \le 8\} \cup\{M_i, 1 \le i \le 8\} \cup \{F_i, 1 \le i \le 6\}\) and equations being equations on the faces $x=0$, $x=1$, $y=0$, $y=1$, $z=0$ and $z=1$ correspondingly, since the weights at all other points in \(V\) can be defined uniquely as described in Remark 2.3. It is easy to verify that this system will be 
\[
Mv=0,
\]
where
\[
M=
\begin{pmatrix}
    1 & 0 & -1 & -1 & 0 & 0 & -1 & -2 & -2 & -1 & 1 & 0 & 0 & 0 & 1 & 1 & -1 & 0 & -1 & 0 & 1 & -2 \\
    0 & 1 & -1 & -1 & -2 & -2 & -1 & 0 & 0 & -1 & 0 & 1 & 1 & 1 & 0 & 0 & 0 & -1 & 0 & -1 & -2 & 1 \\
    1 & 0 & -1 & -1 & -2 & -1 & 0 & 0 & -1 & -2 & 1 & 1 & 1 & 0 & 0 & 0 & -1 & 0 & 1 & -2 & -1 & 0 \\
    0 & 1 & -1 & -1 & 0 & -1 & -2 & -2 & -1 & 0 & 0 & 0 & 0 & 1 & 1 & 1 & 0 & -1 & -2 & 1 & 0 & -1 \\
    1 & 0 & -1 & -1 & -1 & -2 & -2 & -1 & 0 & 0 & 0 & 0 & 1 & 1 & 1 & 0 & 1 & -2 & -1 & 0 & -1 & 0 \\
    0 & 1 & -1 & -1 & -1 & 0 & 0 & -1 & -2 & -2 & 1 & 1 & 0 & 0 & 0 & 1 & -2 & 1 & 0 & -1 & 0 & -1
\end{pmatrix}
\]
\\
and
\(
v=(\lambda(T_1),\lambda(T_8),\lambda(M_1),\lambda(M_2),
%\lambda(M_3),\lambda(M_4),\lambda(M_5),\lambda(M_6),\lambda(M_7),
\dots,\lambda(M_8),\lambda(T_2),\lambda(T_3),
%,\lambda(T_4),\lambda(T_5),\lambda(T_6)
\dots,\lambda(T_7),\lambda(F_1),\lambda(F_2),
%,\lambda(F_3),\lambda(F_4),\lambda(F_5),
\dots,\lambda(F_6))^\mathrm{T}.
\)

Reducing the matrix \(M\) to the row echelon form by using the fact that the rows \(r_1,\dots,r_6\) of the matrix satisfy the identity
\(
r_1+r_2=r_3+r_4=r_5+r_6,
\)
gives \(M \sim M'\), where \( M'\) is the matrix given below
\\
\[
\begin{pmatrix}
    1 & 0 & -1 & -1 & 0 & 0 & -1 & -2 & -2 & -1 & 1 & 0 & 0 & 0 & 1 & 1 & -1 & 0 & -1 & 0 & 1 & -2 \\
    0 & 1 & -1 & -1 & -2 & -2 & -1 & 0 & 0 & -1 & 0 & 1 & 1 & 1 & 0 & 0 & 0 & -1 & 0 & -1 & -2 & 1 \\
    1 & 0 & -1 & -1 & -2 & -1 & 0 & 0 & -1 & -2 & 1 & 1 & 1 & 0 & 0 & 0 & -1 & 0 & 1 & -2 & -1 & 0 \\
    0 & 1 & -1 & -1 & 0 & -1 & -2 & -2 & -1 & 0 & 0 & 0 & 0 & 1 & 1 & 1 & 0 & -1 & -2 & 1 & 0 & -1 \\
    0 & 0 & 0 &0 &0 & 0 & 0 &0 &0 & 0 & 0 &0 &0 & 0 & 0 &0 &0 & 0 & 0 &0 &0 & 0\\
    0 & 0 & 0 &0 &0 & 0 & 0 &0 &0 & 0 & 0 &0 &0 & 0 & 0 &0 &0 & 0 & 0 &0 &0 & 0 
\end{pmatrix}.\]
\bigskip

By the Lemma 2.5 we know that there exists an optimal projection cycle on \(U_x \times U_y \times U_z\) that has $\lambda(T_1)>0$, $\lambda(T_8)>0$, $\lambda(M_i) \ge 0$, $1 \le i \le 8$, $\lambda(F_i) \ge 0$, $1 \le i \le 6$, $\lambda(T_i) \le 0$, $2 \le i \le 7$ and $\lambda(F_i)\lambda(F_j)=0$ if $i \neq j$ and $j-i$ is even. Thus, it is sufficient to find the maximum over all minimal projection cycles on \(V\) satisfying  \(\lambda(T_1),\lambda(T_8)>0.\) Thus, by the \textbf{Theorem} 3.1, we can find all such projection cycles by finding all minimal solutions of the system \(M'v=0\) with \(\lambda(T_1),\lambda(T_8)>0\) or in other words, finding all minimal linearly dependent column systems containing the first two columns.

It is not hard to verify with the help of a computer that there are \(123\) minimal column systems containing the first two columns and corresponding to a minimal solution satisfying the above conditions. And corresponding projection cycle-vectors are the ones given below. Note that we give only the weights on the corner points and the points with positive weights. The remaining points and their weights can be reconstructed as explained above.

\subsection*{1) Only $T_i$ terms}
\[
\textbf{1.1}: T_1, T_8, T_2, T_5:\quad \lambda(T_1) = 1,\quad \lambda(T_8) = 1,\quad \lambda(T_2) = -1,\quad \lambda(T_5) = -1
\]
\[
\textbf{1.2}: T_1, T_8, T_3, T_6:\quad \lambda(T_1) = 1,\quad \lambda(T_8) = 1,\quad \lambda(T_3) = -1,\quad \lambda(T_6) = -1
\]
\[
\textbf{1.3}: T_1, T_8, T_4, T_7:\quad \lambda(T_1) = 1,\quad \lambda(T_8) = 1,\quad \lambda(T_4) = -1,\quad \lambda(T_7) = -1
\]
\[
\textbf{1.4}: T_1, T_8, T_2, T_4, T_6:\quad \lambda(T_1) = 2,\quad \lambda(T_8) = 1,\quad \lambda(T_2) = -1,\quad \lambda(T_4) = -1,\quad \lambda(T_6) = -1
\]
\[
\textbf{1.5}: T_1, T_8, T_3, T_5, T_7:\quad \lambda(T_1) = 1,\quad \lambda(T_8) = 2,\quad \lambda(T_3) = -1,\quad \lambda(T_5) = -1,\quad \lambda(T_7) = -1
\]
\subsection*{2) $T_i$ and one $M_j$}
\[
\textbf{2.1}: T_1, T_8, M_1:\quad \lambda(T_1) = 1,\quad \lambda(T_8) = 1,\quad \lambda(M_1) = 1
\]
\[
\textbf{2.2}: T_1, T_8, M_2:\quad \lambda(T_1) = 1,\quad \lambda(T_8) = 1,\quad \lambda(M_2) = 1
\]
\[
\textbf{2.3}: T_1, T_8, M_3, T_2, T_6:\quad \lambda(T_1) = 3,\quad \lambda(T_8) = 2,\quad \lambda(M_3) = 1,\quad \lambda(T_2) = -1,\quad \lambda(T_6) = -2
\]
\[
\textbf{2.4}: T_1, T_8, M_3, T_5, T_7:\quad \lambda(T_1) = 2,\quad \lambda(T_8) = 3,\quad \lambda(M_3) = 1,\quad \lambda(T_5) = -1,\quad \lambda(T_7) = -2
\]
\[
\textbf{2.5}: T_1, T_8, M_3, T_6, T_7:\quad \lambda(T_1) = 2,\quad \lambda(T_8) = 2,\quad \lambda(M_3) = 1,\quad \lambda(T_6) = -1,\quad \lambda(T_7) = -1
\]
\[
\textbf{2.6}: T_1, T_8, M_4, T_2, T_6:\quad \lambda(T_1) = 3,\quad \lambda(T_8) = 2,\quad \lambda(M_4) = 1,\quad \lambda(T_2) = -2,\quad \lambda(T_6) = -1
\]
\[
\textbf{2.7}: T_1, T_8, M_4, T_2, T_7:\quad \lambda(T_1) = 2,\quad \lambda(T_8) = 2,\quad \lambda(M_4) = 1,\quad \lambda(T_2) = -1,\quad \lambda(T_7) = -1
\]
\[
\textbf{2.8}: T_1, T_8, M_4, T_3, T_7:\quad \lambda(T_1) = 2,\quad \lambda(T_8) = 3,\quad \lambda(M_4) = 1,\quad \lambda(T_3) = -1,\quad \lambda(T_7) = -2
\]
\[
\textbf{2.9}: T_1, T_8, M_5, T_2, T_3:\quad \lambda(T_1) = 2,\quad \lambda(T_8) = 2,\quad \lambda(M_5) = 1,\quad \lambda(T_2) = -1,\quad \lambda(T_3) = -1
\]
\[
\textbf{2.10}: T_1, T_8, M_5, T_2, T_4:\quad \lambda(T_1) = 3,\quad \lambda(T_8) = 2,\quad \lambda(M_5) = 1,\quad \lambda(T_2) = -2,\quad \lambda(T_4) = -1
\]
\[
\textbf{2.11}: T_1, T_8, M_5, T_3, T_7:\quad \lambda(T_1) = 2,\quad \lambda(T_8) = 3,\quad \lambda(M_5) = 1,\quad \lambda(T_3) = -2,\quad \lambda(T_7) = -1
\]
\[
\textbf{2.12}: T_1, T_8, M_6, T_2, T_4:\quad \lambda(T_1) = 3,\quad \lambda(T_8) = 2,\quad \lambda(M_6) = 1,\quad \lambda(T_2) = -1,\quad \lambda(T_4) = -2
\]
\[
\textbf{2.13}: T_1, T_8, M_6, T_3, T_4:\quad \lambda(T_1) = 2,\quad \lambda(T_8) = 2,\quad \lambda(M_6) = 1,\quad \lambda(T_3) = -1,\quad \lambda(T_4) = -1
\]
\[
\textbf{2.14}: T_1, T_8, M_6, T_3, T_5:\quad \lambda(T_1) = 2,\quad \lambda(T_8) = 3,\quad \lambda(M_6) = 1,\quad \lambda(T_3) = -2,\quad \lambda(T_5) = -1
\]
\[
\textbf{2.15}: T_1, T_8, M_7, T_3, T_5:\quad \lambda(T_1) = 2,\quad \lambda(T_8) = 3,\quad \lambda(M_7) = 1,\quad \lambda(T_3) = -1,\quad \lambda(T_5) = -2
\]
\[
\textbf{2.16}: T_1, T_8, M_7, T_4, T_5:\quad \lambda(T_1) = 2,\quad \lambda(T_8) = 2,\quad \lambda(M_7) = 1,\quad \lambda(T_4) = -1,\quad \lambda(T_5) = -1
\]
\[
\textbf{2.17}: T_1, T_8, M_7, T_4, T_6:\quad \lambda(T_1) = 3,\quad \lambda(T_8) = 2,\quad \lambda(M_7) = 1,\quad \lambda(T_4) = -2,\quad \lambda(T_6) = -1
\]
\[
\textbf{2.18}: T_1, T_8, M_8, T_4, T_6:\quad \lambda(T_1) = 3,\quad \lambda(T_8) = 2,\quad \lambda(M_8) = 1,\quad \lambda(T_4) = -1,\quad \lambda(T_6) = -2
\]
\[
\textbf{2.19}: T_1, T_8, M_8, T_5, T_6:\quad \lambda(T_1) = 2,\quad \lambda(T_8) = 2,\quad \lambda(M_8) = 1,\quad \lambda(T_5) = -1,\quad \lambda(T_6) = -1
\]
\[
\textbf{2.20}: T_1, T_8, M_8, T_5, T_7:\quad \lambda(T_1) = 2,\quad \lambda(T_8) = 3,\quad \lambda(M_8) = 1,\quad \lambda(T_5) = -2,\quad \lambda(T_7) = -1
\]
\subsection*{3) $T_i$ and two $M_j$}
\[
\textbf{3.1}: T_1, T_8, M_3, M_6:\quad \lambda(T_1) = 2,\quad \lambda(T_8) = 2,\quad \lambda(M_3) = 1,\quad \lambda(M_6) = 1
\]
\[
\textbf{3.2}: T_1, T_8, M_4, M_7:\quad \lambda(T_1) = 2,\quad \lambda(T_8) = 2,\quad \lambda(M_4) = 1,\quad \lambda(M_7) = 1
\]
\[
\textbf{3.3}: T_1, T_8, M_5, M_8:\quad \lambda(T_1) = 2,\quad \lambda(T_8) = 2,\quad \lambda(M_5) = 1,\quad \lambda(M_8) = 1
\]
\[
\textbf{3.4}: T_1, T_8, M_3, M_4, T_7:\quad \lambda(T_1) = 3,\quad \lambda(T_8) = 4,\quad \lambda(M_3) = 1,\quad \lambda(M_4) = 1,\quad \lambda(T_7) = -3
\]
\[
\textbf{3.5}: T_1, T_8, M_3, M_5, T_2:\quad \lambda(T_1) = 5,\quad \lambda(T_8) = 4,\quad \lambda(M_3) = 1,\quad \lambda(M_5) = 2,\quad \lambda(T_2) = -3
\]
\[
\textbf{3.6}: T_1, T_8, M_3, M_5, T_7:\quad \lambda(T_1) = 4,\quad \lambda(T_8) = 5,\quad \lambda(M_3) = 2,\quad \lambda(M_5) = 1,\quad \lambda(T_7) = -3
\]
\[
\textbf{3.7}: T_1, T_8, M_3, M_7, T_5:\quad \lambda(T_1) = 4,\quad \lambda(T_8) = 5,\quad \lambda(M_3) = 1,\quad \lambda(M_7) = 2,\quad \lambda(T_5) = -3
\]
\[
\textbf{3.8}: T_1, T_8, M_3, M_7, T_6:\quad \lambda(T_1) = 5,\quad \lambda(T_8) = 4,\quad \lambda(M_3) = 2,\quad \lambda(M_7) = 1,\quad \lambda(T_6) = -3
\]
\[
\textbf{3.9}: T_1, T_8, M_3, M_8, T_6:\quad \lambda(T_1) = 4,\quad \lambda(T_8) = 3,\quad \lambda(M_3) = 1,\quad \lambda(M_8) = 1,\quad \lambda(T_6) = -3
\]
\[
\textbf{3.10}: T_1, T_8, M_4, M_5, T_2:\quad \lambda(T_1) = 4,\quad \lambda(T_8) = 3,\quad \lambda(M_4) = 1,\quad \lambda(M_5) = 1,\quad \lambda(T_2) = -3
\]
\[
\textbf{3.11}: T_1, T_8, M_4, M_6, T_2:\quad \lambda(T_1) = 5,\quad \lambda(T_8) = 4,\quad \lambda(M_4) = 2,\quad \lambda(M_6) = 1,\quad \lambda(T_2) = -3
\]
\[
\textbf{3.12}: T_1, T_8, M_4, M_6, T_3:\quad \lambda(T_1) = 4,\quad \lambda(T_8) = 5,\quad \lambda(M_4) = 1,\quad \lambda(M_6) = 2,\quad \lambda(T_3) = -3
\]
\[
\textbf{3.13}: T_1, T_8, M_4, M_8, T_6:\quad \lambda(T_1) = 5,\quad \lambda(T_8) = 4,\quad \lambda(M_4) = 1,\quad \lambda(M_8) = 2,\quad \lambda(T_6) = -3
\]
\[
\textbf{3.14}: T_1, T_8, M_4, M_8, T_7:\quad \lambda(T_1) = 4,\quad \lambda(T_8) = 5,\quad \lambda(M_4) = 2,\quad \lambda(M_8) = 1,\quad \lambda(T_7) = -3
\]
\[
\textbf{3.15}: T_1, T_8, M_5, M_6, T_3:\quad \lambda(T_1) = 3,\quad \lambda(T_8) = 4,\quad \lambda(M_5) = 1,\quad \lambda(M_6) = 1,\quad \lambda(T_3) = -3
\]
\[
\textbf{3.16}: T_1, T_8, M_5, M_7, T_3:\quad \lambda(T_1) = 4,\quad \lambda(T_8) = 5,\quad \lambda(M_5) = 2,\quad \lambda(M_7) = 1,\quad \lambda(T_3) = -3
\]
\[
\textbf{3.17}: T_1, T_8, M_5, M_7, T_4:\quad \lambda(T_1) = 5,\quad \lambda(T_8) = 4,\quad \lambda(M_5) = 1,\quad \lambda(M_7) = 2,\quad \lambda(T_4) = -3
\]
\[
\textbf{3.18}: T_1, T_8, M_6, M_7, T_4:\quad \lambda(T_1) = 4,\quad \lambda(T_8) = 3,\quad \lambda(M_6) = 1,\quad \lambda(M_7) = 1,\quad \lambda(T_4) = -3
\]
\[
\textbf{3.19}: T_1, T_8, M_6, M_8, T_4:\quad \lambda(T_1) = 5,\quad \lambda(T_8) = 4,\quad \lambda(M_6) = 2,\quad \lambda(M_8) = 1,\quad \lambda(T_4) = -3
\]
\[
\textbf{3.20}: T_1, T_8, M_6, M_8, T_5:\quad \lambda(T_1) = 4,\quad \lambda(T_8) = 5,\quad \lambda(M_6) = 1,\quad \lambda(M_8) = 2,\quad \lambda(T_5) = -3
\]
\[
\textbf{3.21}: T_1, T_8, M_7, M_8, T_5:\quad \lambda(T_1) = 3,\quad \lambda(T_8) = 4,\quad \lambda(M_7) = 1,\quad \lambda(M_8) = 1,\quad \lambda(T_5) = -3
\]
\subsection*{4) $T_i$ and three $M_j$}
\[
\textbf{4.1}: T_1, T_8, M_3, M_5, M_7:\quad \lambda(T_1) = 3,\quad \lambda(T_8) = 3,\quad \lambda(M_3) = 1,\quad \lambda(M_5) = 1,\quad \lambda(M_7) = 1
\]
\[
\textbf{4.2}: T_1, T_8, M_4, M_6, M_8:\quad \lambda(T_1) = 3,\quad \lambda(T_8) = 3,\quad \lambda(M_4) = 1,\quad \lambda(M_6) = 1,\quad \lambda(M_8) = 1
\]
\subsection*{5) $T_i$ and one $F_k$}
\[
\textbf{5.1}: T_1, T_8, T_2, F_2:\quad \lambda(T_1) = 2,\quad \lambda(T_8) = 1,\quad \lambda(T_2) = -2,\quad \lambda(F_2) = 1
\]
\[
\textbf{5.2}: T_1, T_8, T_3, F_3:\quad \lambda(T_1) = 1,\quad \lambda(T_8) = 2,\quad \lambda(T_3) = -2,\quad \lambda(F_3) = 1
\]
\[
\textbf{5.3}: T_1, T_8, T_4, F_6:\quad \lambda(T_1) = 2,\quad \lambda(T_8) = 1,\quad \lambda(T_4) = -2,\quad \lambda(F_6) = 1
\]
\[
\textbf{5.4}: T_1, T_8, T_5, F_1:\quad \lambda(T_1) = 1,\quad \lambda(T_8) = 2,\quad \lambda(T_5) = -2,\quad \lambda(F_1) = 1
\]
\[
\textbf{5.5}: T_1, T_8, T_6, F_4:\quad \lambda(T_1) = 2,\quad \lambda(T_8) = 1,\quad \lambda(T_6) = -2,\quad \lambda(F_4) = 1
\]
\[
\textbf{5.6}: T_1, T_8, T_7, F_5:\quad \lambda(T_1) = 1,\quad \lambda(T_8) = 2,\quad \lambda(T_7) = -2,\quad \lambda(F_5) = 1
\]
\[
\textbf{5.7}: T_1, T_8, T_2, T_4, F_3:\quad \lambda(T_1) = 3,\quad \lambda(T_8) = 2,\quad \lambda(T_2) = -2,\quad \lambda(T_4) = -2,\quad \lambda(F_3) = 1
\]
\[
\textbf{5.8}: T_1, T_8, T_2, T_6, F_5:\quad \lambda(T_1) = 3,\quad \lambda(T_8) = 2,\quad \lambda(T_2) = -2,\quad \lambda(T_6) = -2,\quad \lambda(F_5) = 1
\]
\[
\textbf{5.9}: T_1, T_8, T_3, T_5, F_6:\quad \lambda(T_1) = 2,\quad \lambda(T_8) = 3,\quad \lambda(T_3) = -2,\quad \lambda(T_5) = -2,\quad \lambda(F_6) = 1
\]
\[
\textbf{5.10}: T_1, T_8, T_3, T_7, F_2:\quad \lambda(T_1) = 2,\quad \lambda(T_8) = 3,\quad \lambda(T_3) = -2,\quad \lambda(T_7) = -2,\quad \lambda(F_2) = 1
\]
\[
\textbf{5.11}: T_1, T_8, T_4, T_6, F_1:\quad \lambda(T_1) = 3,\quad \lambda(T_8) = 2,\quad \lambda(T_4) = -2,\quad \lambda(T_6) = -2,\quad \lambda(F_1) = 1
\]
\[
\textbf{5.12}: T_1, T_8, T_5, T_7, F_4:\quad \lambda(T_1) = 2,\quad \lambda(T_8) = 3,\quad \lambda(T_5) = -2,\quad \lambda(T_7) = -2,\quad \lambda(F_4) = 1
\]
\subsection*{6) $T_i$, one $M_j$, one $F_k$}
\[
\textbf{6.1}: T_1, T_8, M_3, T_2, F_3:\quad \lambda(T_1) = 2,\quad \lambda(T_8) = 2,\quad \lambda(M_3) = 1,\quad \lambda(T_2) = -1,\quad \lambda(F_3) = 1
\]
\[
\textbf{6.2}: T_1, T_8, M_3, T_5, F_6:\quad \lambda(T_1) = 2,\quad \lambda(T_8) = 2,\quad \lambda(M_3) = 1,\quad \lambda(T_5) = -1,\quad \lambda(F_6) = 1
\]
\[
\textbf{6.3}: T_1, T_8, M_3, T_6, F_1:\quad \lambda(T_1) = 5,\quad \lambda(T_8) = 4,\quad \lambda(M_3) = 2,\quad \lambda(T_6) = -4,\quad \lambda(F_1) = 1
\]
\[
\textbf{6.4}: T_1, T_8, M_3, T_6, F_6:\quad \lambda(T_1) = 4,\quad \lambda(T_8) = 3,\quad \lambda(M_3) = 2,\quad \lambda(T_6) = -2,\quad \lambda(F_6) = 1
\]
\[
\textbf{6.5}: T_1, T_8, M_3, T_7, F_2:\quad \lambda(T_1) = 4,\quad \lambda(T_8) = 5,\quad \lambda(M_3) = 2,\quad \lambda(T_7) = -4,\quad \lambda(F_2) = 1
\]
\[
\textbf{6.6}: T_1, T_8, M_3, T_7, F_3:\quad \lambda(T_1) = 3,\quad \lambda(T_8) = 4,\quad \lambda(M_3) = 2,\quad \lambda(T_7) = -2,\quad \lambda(F_3) = 1
\]
\[
\textbf{6.7}: T_1, T_8, M_4, T_2, F_3:\quad \lambda(T_1) = 5,\quad \lambda(T_8) = 4,\quad \lambda(M_4) = 2,\quad \lambda(T_2) = -4,\quad \lambda(F_3) = 1
\]
\[
\textbf{6.8}: T_1, T_8, M_4, T_2, F_6:\quad \lambda(T_1) = 4,\quad \lambda(T_8) = 3,\quad \lambda(M_4) = 2,\quad \lambda(T_2) = -2,\quad \lambda(F_6) = 1
\]
\[
\textbf{6.9}: T_1, T_8, M_4, T_3, F_6:\quad \lambda(T_1) = 2,\quad \lambda(T_8) = 2,\quad \lambda(M_4) = 1,\quad \lambda(T_3) = -1,\quad \lambda(F_6) = 1
\]
\[
\textbf{6.10}: T_1, T_8, M_4, T_6, F_1:\quad \lambda(T_1) = 2,\quad \lambda(T_8) = 2,\quad \lambda(M_4) = 1,\quad \lambda(T_6) = -1,\quad \lambda(F_1) = 1
\]
\[
\textbf{6.11}: T_1, T_8, M_4, T_7, F_1:\quad \lambda(T_1) = 3,\quad \lambda(T_8) = 4,\quad \lambda(M_4) = 2,\quad \lambda(T_7) = -2,\quad \lambda(F_1) = 1
\]
\[
\textbf{6.12}: T_1, T_8, M_4, T_7, F_4:\quad \lambda(T_1) = 4,\quad \lambda(T_8) = 5,\quad \lambda(M_4) = 2,\quad \lambda(T_7) = -4,\quad \lambda(F_4) = 1
\]
\[
\textbf{6.13}: T_1, T_8, M_5, T_2, F_4:\quad \lambda(T_1) = 4,\quad \lambda(T_8) = 3,\quad \lambda(M_5) = 2,\quad \lambda(T_2) = -2,\quad \lambda(F_4) = 1
\]
\[
\textbf{6.14}: T_1, T_8, M_5, T_2, F_5:\quad \lambda(T_1) = 5,\quad \lambda(T_8) = 4,\quad \lambda(M_5) = 2,\quad \lambda(T_2) = -4,\quad \lambda(F_5) = 1
\]
\[
\textbf{6.15}: T_1, T_8, M_5, T_3, F_1:\quad \lambda(T_1) = 3,\quad \lambda(T_8) = 4,\quad \lambda(M_5) = 2,\quad \lambda(T_3) = -2,\quad \lambda(F_1) = 1
\]
\[
\textbf{6.16}: T_1, T_8, M_5, T_3, F_6:\quad \lambda(T_1) = 4,\quad \lambda(T_8) = 5,\quad \lambda(M_5) = 2,\quad \lambda(T_3) = -4,\quad \lambda(F_6) = 1
\]
\[
\textbf{6.17}: T_1, T_8, M_5, T_4, F_1:\quad \lambda(T_1) = 2,\quad \lambda(T_8) = 2,\quad \lambda(M_5) = 1,\quad \lambda(T_4) = -1,\quad \lambda(F_1) = 1
\]
\[
\textbf{6.18}: T_1, T_8, M_5, T_7, F_4:\quad \lambda(T_1) = 2,\quad \lambda(T_8) = 2,\quad \lambda(M_5) = 1,\quad \lambda(T_7) = -1,\quad \lambda(F_4) = 1
\]
\[
\textbf{6.19}: T_1, T_8, M_6, T_2, F_5:\quad \lambda(T_1) = 2,\quad \lambda(T_8) = 2,\quad \lambda(M_6) = 1,\quad \lambda(T_2) = -1,\quad \lambda(F_5) = 1
\]
\[
\textbf{6.20}: T_1, T_8, M_6, T_3, F_2:\quad \lambda(T_1) = 4,\quad \lambda(T_8) = 5,\quad \lambda(M_6) = 2,\quad \lambda(T_3) = -4,\quad \lambda(F_2) = 1
\]
\[
\textbf{6.21}: T_1, T_8, M_6, T_3, F_5:\quad \lambda(T_1) = 3,\quad \lambda(T_8) = 4,\quad \lambda(M_6) = 2,\quad \lambda(T_3) = -2,\quad \lambda(F_5) = 1
\]
\[
\textbf{6.22}: T_1, T_8, M_6, T_4, F_1:\quad \lambda(T_1) = 5,\quad \lambda(T_8) = 4,\quad \lambda(M_6) = 2,\quad \lambda(T_4) = -4,\quad \lambda(F_1) = 1
\]
\[
\textbf{6.23}: T_1, T_8, M_6, T_4, F_4:\quad \lambda(T_1) = 4,\quad \lambda(T_8) = 3,\quad \lambda(M_6) = 2,\quad \lambda(T_4) = -2,\quad \lambda(F_4) = 1
\]
\[
\textbf{6.24}: T_1, T_8, M_6, T_5, F_4:\quad \lambda(T_1) = 2,\quad \lambda(T_8) = 2,\quad \lambda(M_6) = 1,\quad \lambda(T_5) = -1,\quad \lambda(F_4) = 1
\]
\[
\textbf{6.25}: T_1, T_8, M_7, T_3, F_2:\quad \lambda(T_1) = 2,\quad \lambda(T_8) = 2,\quad \lambda(M_7) = 1,\quad \lambda(T_3) = -1,\quad \lambda(F_2) = 1
\]
\[
\textbf{6.26}: T_1, T_8, M_7, T_4, F_2:\quad \lambda(T_1) = 4,\quad \lambda(T_8) = 3,\quad \lambda(M_7) = 2,\quad \lambda(T_4) = -2,\quad \lambda(F_2) = 1
\]
\[
\textbf{6.27}: T_1, T_8, M_7, T_4, F_3:\quad \lambda(T_1) = 5,\quad \lambda(T_8) = 4,\quad \lambda(M_7) = 2,\quad \lambda(T_4) = -4,\quad \lambda(F_3) = 1
\]
\[
\textbf{6.28}: T_1, T_8, M_7, T_5, F_4:\quad \lambda(T_1) = 4,\quad \lambda(T_8) = 5,\quad \lambda(M_7) = 2,\quad \lambda(T_5) = -4,\quad \lambda(F_4) = 1
\]
\[
\textbf{6.29}: T_1, T_8, M_7, T_5, F_5:\quad \lambda(T_1) = 3,\quad \lambda(T_8) = 4,\quad \lambda(M_7) = 2,\quad \lambda(T_5) = -2,\quad \lambda(F_5) = 1
\]
\[
\textbf{6.30}: T_1, T_8, M_7, T_6, F_5:\quad \lambda(T_1) = 2,\quad \lambda(T_8) = 2,\quad \lambda(M_7) = 1,\quad \lambda(T_6) = -1,\quad \lambda(F_5) = 1
\]
\[
\textbf{6.31}: T_1, T_8, M_8, T_4, F_3:\quad \lambda(T_1) = 2,\quad \lambda(T_8) = 2,\quad \lambda(M_8) = 1,\quad \lambda(T_4) = -1,\quad \lambda(F_3) = 1
\]
\[
\textbf{6.32}: T_1, T_8, M_8, T_5, F_3:\quad \lambda(T_1) = 3,\quad \lambda(T_8) = 4,\quad \lambda(M_8) = 2,\quad \lambda(T_5) = -2,\quad \lambda(F_3) = 1
\]
\[
\textbf{6.33}: T_1, T_8, M_8, T_5, F_6:\quad \lambda(T_1) = 4,\quad \lambda(T_8) = 5,\quad \lambda(M_8) = 2,\quad \lambda(T_5) = -4,\quad \lambda(F_6) = 1
\]
\[
\textbf{6.34}: T_1, T_8, M_8, T_6, F_2:\quad \lambda(T_1) = 4,\quad \lambda(T_8) = 3,\quad \lambda(M_8) = 2,\quad \lambda(T_6) = -2,\quad \lambda(F_2) = 1
\]
\[
\textbf{6.35}: T_1, T_8, M_8, T_6, F_5:\quad \lambda(T_1) = 5,\quad \lambda(T_8) = 4,\quad \lambda(M_8) = 2,\quad \lambda(T_6) = -4,\quad \lambda(F_5) = 1
\]
\[
\textbf{6.36}: T_1, T_8, M_8, T_7, F_2:\quad \lambda(T_1) = 2,\quad \lambda(T_8) = 2,\quad \lambda(M_8) = 1,\quad \lambda(T_7) = -1,\quad \lambda(F_2) = 1
\]
\subsection*{7) $T_i$, two $M_j$, one $F_k$}
\[
\textbf{7.1}: T_1, T_8, M_3, M_4, F_6:\quad \lambda(T_1) = 6,\quad \lambda(T_8) = 5,\quad \lambda(M_3) = 2,\quad \lambda(M_4) = 2,\quad \lambda(F_6) = 3
\]
\[
\textbf{7.2}: T_1, T_8, M_3, M_5, F_1:\quad \lambda(T_1) = 7,\quad \lambda(T_8) = 8,\quad \lambda(M_3) = 2,\quad \lambda(M_5) = 4,\quad \lambda(F_1) = 3
\]
\[
\textbf{7.3}: T_1, T_8, M_3, M_5, F_6:\quad \lambda(T_1) = 8,\quad \lambda(T_8) = 7,\quad \lambda(M_3) = 4,\quad \lambda(M_5) = 2,\quad \lambda(F_6) = 3
\]
\[
\textbf{7.4}: T_1, T_8, M_3, M_7, F_2:\quad \lambda(T_1) = 8,\quad \lambda(T_8) = 7,\quad \lambda(M_3) = 2,\quad \lambda(M_7) = 4,\quad \lambda(F_2) = 3
\]
\[
\textbf{7.5}: T_1, T_8, M_3, M_7, F_3:\quad \lambda(T_1) = 7,\quad \lambda(T_8) = 8,\quad \lambda(M_3) = 4,\quad \lambda(M_7) = 2,\quad \lambda(F_3) = 3
\]
\[
\textbf{7.6}: T_1, T_8, M_3, M_8, F_3:\quad \lambda(T_1) = 5,\quad \lambda(T_8) = 6,\quad \lambda(M_3) = 2,\quad \lambda(M_8) = 2,\quad \lambda(F_3) = 3
\]
\[
\textbf{7.7}: T_1, T_8, M_4, M_5, F_1:\quad \lambda(T_1) = 5,\quad \lambda(T_8) = 6,\quad \lambda(M_4) = 2,\quad \lambda(M_5) = 2,\quad \lambda(F_1) = 3
\]
\[
\textbf{7.8}: T_1, T_8, M_4, M_6, F_1:\quad \lambda(T_1) = 7,\quad \lambda(T_8) = 8,\quad \lambda(M_4) = 4,\quad \lambda(M_6) = 2,\quad \lambda(F_1) = 3
\]
\[
\textbf{7.9}: T_1, T_8, M_4, M_6, F_4:\quad \lambda(T_1) = 8,\quad \lambda(T_8) = 7,\quad \lambda(M_4) = 2,\quad \lambda(M_6) = 4,\quad \lambda(F_4) = 3
\]
\[
\textbf{7.10}: T_1, T_8, M_4, M_8, F_3:\quad \lambda(T_1) = 7,\quad \lambda(T_8) = 8,\quad \lambda(M_4) = 2,\quad \lambda(M_8) = 4,\quad \lambda(F_3) = 3
\]
\[
\textbf{7.11}: T_1, T_8, M_4, M_8, F_6:\quad \lambda(T_1) = 8,\quad \lambda(T_8) = 7,\quad \lambda(M_4) = 4,\quad \lambda(M_8) = 2,\quad \lambda(F_6) = 3
\]
\[
\textbf{7.12}: T_1, T_8, M_5, M_6, F_4:\quad \lambda(T_1) = 6,\quad \lambda(T_8) = 5,\quad \lambda(M_5) = 2,\quad \lambda(M_6) = 2,\quad \lambda(F_4) = 3
\]
\[
\textbf{7.13}: T_1, T_8, M_5, M_7, F_4:\quad \lambda(T_1) = 8,\quad \lambda(T_8) = 7,\quad \lambda(M_5) = 4,\quad \lambda(M_7) = 2,\quad \lambda(F_4) = 3
\]
\[
\textbf{7.14}: T_1, T_8, M_5, M_7, F_5:\quad \lambda(T_1) = 7,\quad \lambda(T_8) = 8,\quad \lambda(M_5) = 2,\quad \lambda(M_7) = 4,\quad \lambda(F_5) = 3
\]
\[
\textbf{7.15}: T_1, T_8, M_6, M_7, F_5:\quad \lambda(T_1) = 5,\quad \lambda(T_8) = 6,\quad \lambda(M_6) = 2,\quad \lambda(M_7) = 2,\quad \lambda(F_5) = 3
\]
\[
\textbf{7.16}: T_1, T_8, M_6, M_8, F_2:\quad \lambda(T_1) = 8,\quad \lambda(T_8) = 7,\quad \lambda(M_6) = 2,\quad \lambda(M_8) = 4,\quad \lambda(F_2) = 3
\]
\[
\textbf{7.17}: T_1, T_8, M_6, M_8, F_5:\quad \lambda(T_1) = 7,\quad \lambda(T_8) = 8,\quad \lambda(M_6) = 4,\quad \lambda(M_8) = 2,\quad \lambda(F_5) = 3
\]
\[
\textbf{7.18}: T_1, T_8, M_7, M_8, F_2:\quad \lambda(T_1) = 6,\quad \lambda(T_8) = 5,\quad \lambda(M_7) = 2,\quad \lambda(M_8) = 2,\quad \lambda(F_2) = 3
\]
\subsection*{8) $T_i$ and two $F_k$}
\[
\textbf{8.1}: T_1, T_8, F_1, F_2:\quad \lambda(T_1) = 1,\quad \lambda(T_8) = 1,\quad \lambda(F_1) = 1,\quad \lambda(F_2) = 1
\]
\[
\textbf{8.2}: T_1, T_8, F_3, F_4:\quad \lambda(T_1) = 1,\quad \lambda(T_8) = 1,\quad \lambda(F_3) = 1,\quad \lambda(F_4) = 1
\]
\[
\textbf{8.3}: T_1, T_8, F_5, F_6:\quad \lambda(T_1) = 1,\quad \lambda(T_8) = 1,\quad \lambda(F_5) = 1,\quad \lambda(F_6) = 1
\]

\subsection*{9) $T_i$, one $M_j$, two $F_k$}

\[
\textbf{9.1}: T_1, T_8, M_3, F_3, F_6:\quad \lambda(T_1) = 3,\quad \lambda(T_8) = 3,\quad \lambda(M_3) = 2,\quad \lambda(F_3) = 1,\quad \lambda(F_6) = 1
\]
\[
\textbf{9.2}: T_1, T_8, M_4, F_1, F_6:\quad \lambda(T_1) = 3,\quad \lambda(T_8) = 3,\quad \lambda(M_4) = 2,\quad \lambda(F_1) = 1,\quad \lambda(F_6) = 1
\]
\[
\textbf{9.3}: T_1, T_8, M_5, F_1, F_4:\quad \lambda(T_1) = 3,\quad \lambda(T_8) = 3,\quad \lambda(M_5) = 2,\quad \lambda(F_1) = 1,\quad \lambda(F_4) = 1
\]
\[
\textbf{9.4}: T_1, T_8, M_6, F_4, F_5:\quad \lambda(T_1) = 3,\quad \lambda(T_8) = 3,\quad \lambda(M_6) = 2,\quad \lambda(F_4) = 1,\quad \lambda(F_5) = 1
\]
\[
\textbf{9.5}: T_1, T_8, M_7, F_2, F_5:\quad \lambda(T_1) = 3,\quad \lambda(T_8) = 3,\quad \lambda(M_7) = 2,\quad \lambda(F_2) = 1,\quad \lambda(F_5) = 1
\]
\[
\textbf{9.6}: T_1, T_8, M_8, F_2, F_3:\quad \lambda(T_1) = 3,\quad \lambda(T_8) = 3,\quad \lambda(M_8) = 2,\quad \lambda(F_2) = 1,\quad \lambda(F_3) = 1
\]
\begin{remark}
    For the above minimal projection cycles $(p,\lambda)$, there exists a function $f$, satisfying conditions (2.1), such that $E(f; \Omega)=\frac{\sum \lambda_i f(x_i)}{\sum \lambda_i}$. For example, \\ for the minimal projection cycle \\    
    1.1. $T_1, T_8, T_2, T_5: \ \lambda(T_1) = 1, \ \lambda(T_8) = 1, \ \lambda(T_2) = -1, \ \lambda(T_5) = -1$ \\
this is the function $f(x,y,z)=xz$:
    \[ E(f, \Omega)=\frac{1}{4}=\frac{\sum_{i=1,8,2,5}\lambda(T_i)f(T_i)}{ \sum_{i=1,8,2,5}|\lambda(T_i)|};\]
for the minimal projection cycle \\    
    1.5. $T_1, T_8, T_3, T_5, T_7:\quad \lambda(T_1) = 1,\quad \lambda(T_8) = 2,\quad \lambda(T_3) = -1,\quad \lambda(T_5) = -1,\quad \lambda(T_7) = -1$ \\
this is the function $f(x,y,z)=xyz$:
    \[ E(f, \Omega)=\frac{1}{3}=\frac{\sum_{i=1,8,3,5,7}\lambda(T_i)f(T_i)}{ \sum_{i=1,8,3,5,7}|\lambda(T_i)|};\]
for the minimal projection cycle \\    
    2.11. $T_1, T_8, M_5, T_3, T_7:\quad \lambda(T_1) = 2,\quad \lambda(T_8) = 3,\quad \lambda(M_5) = 1,\quad \lambda(T_3) = -2,\quad \lambda(T_7) = -1$ \\
this is the function $f(x,y,z)= \begin{cases}
        xz, \ \ \ \ \ \ \ \ \ \ \ \ \ \ \ \ \ \ \ \ \ \ \ \ for \ \  0 \le x \le \frac{1}{2}, \ 0 \le z \le \frac{1}{2}, \\
        \frac{z}{2}+\frac{(2x-1)y}{4}, \ \ \ \ \ \ \ \ \ \ \ \ \ for \ \  \frac{1}{2} \le x \le 1, \ 0 \le z \le \frac{1}{2}, \\
        \frac{x}{2}+\frac{(2z-1)y}{4}, \ \ \ \ \ \ \ \ \ \ \ \ \ for \ \  0 \le x \le \frac{1}{2}, \ \frac{1}{2} \le z \le 1, \\
        \frac{1}{4}+\frac{(2x-1)y}{4}+\frac{(2z-1)y}{4}, \ for \ \  \frac{1}{2} \le x \le 1, \ \frac{1}{2} \le z \le 1
    \end{cases}$:
    \[ E(f, \Omega)=\frac{7}{48}=\frac{\sum_{i=1,8,5,7}\lambda(T_i)f(T_i)+\lambda(M_5)f(M_5)+\sum_{i=3,4,5}\lambda(M_{5,l_i})f(M_{5,l_i})}{ \sum_{i=1,8,5,7}|\lambda(T_i)|+|\lambda(M_5)|+\sum_{i=3,4,5}|\lambda(M_{5,l_i})|},\]
where $M_5=(\frac{1}{2},\frac{1}{2},\frac{1}{2})$, $M_{5,l_i}=\Pr_{l_i}M_5$, $i=3,4,5$.
\end{remark}

\end{document}